\documentclass[11pt]{article}
\usepackage{comment}
\usepackage{amsthm,amstext,amssymb,amsmath,graphicx,amsfonts,caption,subcaption,multirow,epigraph,mathrsfs}
\usepackage[all]{xy}
\usepackage[utf8]{inputenc}
\usepackage{booktabs}
\usepackage{tikz}
\usetikzlibrary{matrix}
\usetikzlibrary{arrows,topaths}
\usepackage{placeins}
\usepackage{mathrsfs}
\usepackage{multicol}
\usepackage{natbib}
\usepackage[pagebackref=false,colorlinks,linkcolor=blue,citecolor=blue]{hyperref}
\usepackage{tikz-cd}
\usepackage{xcolor,soul}
\theoremstyle{plain}
\newtheorem{theorem}{Theorem}[section]
\newtheorem{lemma}[theorem]{Lemma}
\newtheorem{proposition}[theorem]{Proposition}

\theoremstyle{definition}
\newtheorem{definition}[theorem]{Definition}
\newtheorem{corollary}[theorem]{Corollary}
\newtheorem{example}[theorem]{Example}

\theoremstyle{remark}
\newtheorem{remark}{\sc Remark}
\makeatother
\makeatletter
\def\namedlabel#1#2{\begingroup
   \def\@currentlabel{#2}%
   \label{#1}\endgroup}
\makeatother
\date{}
\title{\bf  \'{E}tal\'{e} spaces of residuated lattices}\vspace{.25 in}
\author{ \vspace{.25 in} {\bf Saeed Rasouli$^{1*}$} and {\bf Seyed Naser Hosseini$^2$} and {\bf Amin Dehghani$^3$}\\
$^1,^3$Department of Mathematics, Persian Gulf University, Bushehr, Iran\\
$^2$Department of mathematics, Shahid Bahonar University, Kerman,Iran\\
{\tt $^1$srasouli@pgu.ac.ir }\\
{\tt $^2$nhoseini@uk.ac.ir }\\
{\tt $^3$dehghany.amin@hotmail.com}\\}

\def\rarm#1#2#3{$\xymatrix{#1 : #2 \ar@{ >->}[r] & #3}$}
\def\inc#1#2#3{$\xymatrix{#1: #2 \ar@{^(->}[r] & #3}$}

\def\rarmt#1#2#3{$\xymatrix{#1 \ar@{ >->}[r]^{#2} & #3}$}

\def\rars#1#2#3{$\xymatrix{#1 \ar@{|->}[r]^{#2} & #3}$}
\def\triangle#1#2#3#4#5#6{%
$\xymatrix{& #1\ar@{ >->}[d]^{#5} \\ #2\ar[ur]^{#4} \ar[r]_{#6}& #3}$
}%
\def\square#1#2#3#4#5#6#7#8#9{%
$\xymatrix{#1 \ar[rr]^{#5}\ar[d]_{#6} \ar@{}[rd]|{\hs{16}#9} && #2 \ar[d]^{#7} \crcr #3 \ar[rr]_{#8} && #4}$}
\def\,{\hskip-1mm ,}
\def\.{\hskip-1mm .}

\def\dblp#1{(\hspace{-0.7mm}( #1 )\hspace{-0.7mm})}

\begin{document}
 \maketitle
 \begin{abstract}
This paper explores the concept of \'{e}tal\'{e} spaces associated with residuated lattices. Notions of bundles and \'{e}tal\'{e}s of residuated lattices over a given topological space are introduced and investigated. For a topological space $\mathscr{B}$, we establish that the category of \'{e}tal\'{e}s of residuated lattices over $\mathscr{B}$ with morphisms of \'{e}tal\'{e}s of residuated lattices is coreflective in the category of bundles of residuated lattices over $\mathscr{B}$ along with morphisms of bundles of residuated lattices. We provide a method for transferring an \'{e}tal\'{e} of residuated lattices over a topological space to another, utilizing a continuous map. Finally, we define a contravariant functor, called the section functor, from the category of \'{e}tal\'{e}s of residuated lattices with inverse morphisms to the category of residuated lattices.
\footnote{
$^*$: Corresponding author}
\end{abstract}
\section{Introduction}

In the realm of information processing, particularly inferences derived from certain information, classical logic serves as a foundational system. However, to address the challenges posed by uncertainty and fuzzy information, diverse non-classical logic systems have been extensively proposed and researched \citep{turunen1999mathematics,galatos2007residuated,hajek2013metamathematics}. These non-classical logics have proven indispensable in computer science, offering formal and effective tools for managing uncertain and fuzzy data.

Several logical algebras have been introduced as semantical frameworks to support non-classical logic systems. Examples of such algebras include BCK-algebras, residuated lattices, divisible residuated lattices, MTL-algebras, Girard monoids, BL-algebras, G\"{o}del algebras, MV-algebras, and Heyting algebras. Residuated lattices, in particular, are fundamental and crucial algebraic structures that play an essential role in the theory of fuzzy logic. In the context of residuated lattices, there is an additional binary operation $\to$ (often called the \textit{residual} or \textit{implication}) that relates to the lattice structure. This operation in a residuated lattice aligns with the idea of implication in fuzzy logic, where the degree of truth of an implication is related to the degree of membership of the antecedent and consequent. Indeed, the connection between residuated lattices and fuzzy sets is based on their ability to represent and manage uncertainty and reasoning.

Residuated lattices, on the other hand, are not merely important from a logical standpoint; they also hold intrigue from an algebraic perspective, boasting noteworthy properties. Their versatility and foundational nature make them pivotal in reasoning about uncertain information, contributing significantly to both the logical and algebraic aspects of fuzzy logic theory.

This paper delves into the exploration of \'{e}tal\'{e} spaces associated with residuated lattices. Historically, the foundational developments leading to the introduction of \'{e}tal\'{e} spaces can be traced back to the seminal contributions of the S\'{e}minaire Henri Cartan (1948-1964). In collaboration with eminent mathematicians such as Jean-Pierre Serre and Jean Leray, Henri Cartan played a pivotal role in pioneering innovative concepts and techniques in topology and algebraic geometry throughout this seminar. To be precise, it can be said that the notion of an \textit{``\'{e}tal\'{e} space" or ``Espace \'{e}tal\'{e}}" emanating from the realm of algebraic geometry, was born by Henri Cartan, commonly ascribed to his collaborative efforts with Michel Lazard, exactly in Expos\'{e} 14 of this seminar, under the term ``faisceau" \citep{cartan1950faisceaux}. According to \citet[p. 7]{Gray1979}, Lazzard's faisceau on a regular topological space $X$ is defined by a map $p$ from a set $F$ to $X$, satisfying two conditions:
\begin{enumerate}
  \item for each $x\in X$, $p^{-1}(x)=F_{x}$, is a $K$-module;
  \item $F$ has a (non-separated) topology such that the algebraic operations of $F$ are continuous, and $p$ is a local homeomorphism.
\end{enumerate}

Roger Godement, a notable participant in Henri Cartan's S\'{e}minaire, is credited with introducing the term ``Espace \'{e}tal\'{e}" instead of ``faisceau" in his influential work \citep{godement1958topologie}. His elegant contribution encompasses a comprehensive survey of \'{e}tal\'{e} spaces. Godement's exposition played a pivotal role in establishing the formalization and examination of étalé spaces, exerting a substantial influence on the evolution of algebraic geometry.

The notion of \'{e}tal\'{e} spaces has been extensively studied during past decades, see e.g., \citep{grothendieck1960elements,Tennison1976,Bredon1997,maclane2012sheaves,brown2020sheaf}, and \citep{rosiak2022sheaf}. For more historical notes about \'{e}tal\'{e} spaces, the interested reader can be referred to \citep{Gray1979,fasanelli1981creation,miller2000leray}.

Over the years, \'{e}tal\'{e} spaces have evolved into a fundamental component of modern algebraic geometry, playing a pivotal role in the formulation of schemes—an essential concept in Grothendieck's reformulation of algebraic geometry. Presently, the theory of \'{e}tal\'{e} spaces remains a dynamic and fertile area of research, establishing crucial connections between algebraic geometry and topology, and catalyzing advancements across diverse mathematical disciplines.

The influence of \'{e}tal\'{e} spaces has extended beyond their original domain in algebraic geometry, reaching into various branches of mathematics, including number theory and representation theory. Furthermore, the integration of \'{e}tal\'{e} spaces into Fuzzy Sheaf Theory has proven to be particularly fruitful, offering valuable insights into challenges related to spatial reasoning, knowledge representation, and fuzzy image analysis.

Inspired by the foundational notion of \'{e}tal\'{e} spaces, numerous authors have proposed analogous concepts under different names for various structures throughout the years. Examples include sheaves of lattices \citep{brezuleanu1969duale}, sheaves of rings, $R$-modules, and abelian groups \citep{hofmann1972representations}, sheaf spaces of universal algebras \citep{davey1973sheaf}, sheaves of chains over Boolean spaces \citep{cignoli1978lattice}, sheaves of normal lattices \citep{georgescu1987isomorphic}, sheaves of MV-algebras \citep{filipoiu1995compact}, sheaves of BL-algebras \citep{di2003compact}, sheaves of Banach spaces \citep{hofmann2006sheaf}, among others. This diversity of applications reflects the enduring impact and versatility of the \'{e}tal\'{e} spaces concept across various mathematical structures and frameworks.

The above discussion motivates us to introduce and investigate the notion of \'{e}tal\'{e} spaces of residuated lattices in a similar fashion to \'{e}tal\'{e} spaces of rings and modules. This motivation is to find a common pattern for some facts about \'{e}tal\'{e} spaces of rings and \'{e}tal\'{e} spaces of residuated lattices which occur in a much similar way in both cases of rings and of residuated lattices. Also, as known, \'{e}tal\'{e} spaces arise from the theory of sheaves, a foundational framework that has significantly advanced our understanding of cohomology and the topology of algebraic varieties. While the majority of the aforementioned explorations delve into \'{e}tal\'{e} spaces through the lens of sheaves, our approach takes a different approach by avoiding their use.

This paper is organized into five sections as follows: In Sec. \ref{sec2}, we revisit essential definitions, properties, and results related to residuated lattices, laying the groundwork for subsequent discussions. Sec. \ref{sec3} delves into the concept of bundles and \'{e}tal\'{e}s of residuated lattices over a given topological space, recalling specific definitions of \'{e}tal\'{e}s of sets that will be utilized in our exploration. For a topological space $\mathscr{B}$, we establish that the category of \'{e}tal\'{e}s of residuated lattices over $\mathscr{B}$ with morphisms of \'{e}tal\'{e}s of residuated lattices, denoted as $\textbf{RL-Etale}(\mathscr{B})$, forms a full subcategory of bundles of residuated lattices over $\mathscr{B}$ along with morphisms of bundles of residuated lattices, denoted as $\textbf{RL-Bundle}(\mathscr{B})$ (Corollary \ref{rletafulsubcatrlbund}). Sect. \ref{sec4} investigates the interrelation between bundles and \'{e}tal\'{e}s of residuated lattices. As a main result, it sharpens the previous result and demonstrates that given a topological space $\mathscr{B}$, $\textbf{RL-Etale}(\mathscr{B})$ is coreflective in $\textbf{RL-Bundle}(\mathscr{B})$ (Theorem \ref{rletacorefrlbund}). Sect. \ref{sec5} provides a method for transferring an \'{e}tal\'{e} from one topological space to another, utilizing a continuous map. At the end of this section, we define a contravariant functor, called the section functor, from the category of \'{e}tal\'{e}s of residuated lattices with inverse morphisms to the category of residuated lattices (Theorem \ref{maintheoremofart}).

\section{Residuated lattices}\label{sec2}
%
%

In this section, we revisit essential definitions, properties, and results related to residuated lattices, providing a foundation for subsequent discussions.

An algebra $\mathfrak{A} = (A; \vee, \wedge, \odot, \rightarrow, 0, 1)$ is termed a residuated lattice if $\ell(\mathfrak{A}) = (A; \vee, \wedge, 0, 1)$ is a bounded lattice, $(A; \odot, 1)$ is a commutative monoid, and $(\odot, \rightarrow)$ forms an adjoint pair. A residuated lattice $\mathfrak{A}$ is considered non-degenerate if $0 \neq 1$. For a residuated lattice $\mathfrak{A}$ and $a \in A$, we define $\neg a := a \rightarrow 0$, and $a^n := a \odot \ldots \odot a$ ($n$ times) for every integer $n$. The class of residuated lattices, denoted as \textbf{RL}, is equational and, therefore, a variety. For an in-depth exploration of residuated lattices, interested readers are directed to \cite{galatos2007residuated}.

\begin{example}\label{exa4}
Consider the lattice $A_4 = \{0, a, b, 1\}$ with the Hasse diagram shown in Figure \ref{figa4}. Routine calculations reveal that $\mathfrak{A}_4 = (A_4; \vee, \wedge, \odot, \rightarrow, 0, 1)$ is a residuated lattice. The commutative operation $\odot$ is defined by Table \ref{taba4}, and the operation $\rightarrow$ is articulated by $x \rightarrow y = \bigvee \{z \in A_4 \mid x \odot z \leq y\}$, for any $x, y \in A_4$.
\FloatBarrier
\begin{table}[ht]
\begin{minipage}[b]{0.56\linewidth}
\centering
\begin{tabular}{ccccccc}
\hline
$\odot$ & 0 & a & b & 1 \\ \hline
0       & 0 & 0 & 0 & 0 \\
        & a & a & 0 & a \\
        &   & b & b & b \\
        &   &   & 1 & 1 \\ \hline
\end{tabular}
\caption{Cayley table for ``$\odot$" of $\mathfrak{A}_4$}
\label{taba4}
\end{minipage}\hfill
\begin{minipage}[b]{0.6\linewidth}
\centering
  \begin{tikzpicture}[>=stealth',semithick,auto]
    \tikzstyle{subj} = [circle, minimum width=6pt, fill, inner sep=0pt]
    \tikzstyle{obj}  = [circle, minimum width=6pt, draw, inner sep=0pt]

    \tikzstyle{every label}=[font=\bfseries]

    \node[subj,  label=below:0] (0) at (0,0) {};
    \node[subj,  label=below:a] (a) at (-1,1) {};
    \node[subj,  label=below:b] (b) at (1,1) {};
    \node[subj,  label=right:1] (1) at (0,2) {};

    \path[-]   (0)    edge                node{}      (a);
    \path[-]   (0)    edge                node{}      (b);
    \path[-]   (a)    edge                node{}      (1);
    \path[-]   (b)    edge                node{}      (1);
\end{tikzpicture}

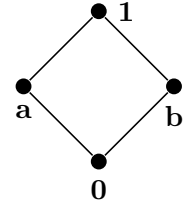
\captionof{figure}{Hasse diagram of $\mathfrak{A}_{4}$}
\label{figa4}
\end{minipage}
\end{table}
\FloatBarrier
\end{example}
\begin{example}\label{exa6}
Consider the lattice $A_6=\{0,a,b,c,d,1\}$ with the Hasse diagram shown in Figure \ref{figa6}. Routine calculations reveal that $\mathfrak{A}_6=(A_6;\vee,\wedge,\odot,\rightarrow,0,1)$ is a residuated lattice. The commutative operation $\odot$ is defined by Table \ref{taba6}, and the operation $\rightarrow$ is articulated by $x \rightarrow y = \bigvee \{z \in A_6 \mid x \odot z \leq y\}$, for any $x, y \in A_6$.
\FloatBarrier
\begin{table}[ht]
\begin{minipage}[b]{0.56\linewidth}
\centering
\begin{tabular}{ccccccc}
\hline
$\odot$ & 0 & a & b & c & d & 1 \\ \hline
0       & 0 & 0 & 0 & 0 & 0 & 0 \\
        & a & a & a & 0 & a & a \\
        &   & b & a & 0 & a & b \\
        &   &   & c & c & c & c \\
        &   &   &   & d & d & d \\
        &   &   &   &   & 1 & 1 \\ \hline
\end{tabular}
\caption{Cayley table for ``$\odot$" of $\mathfrak{A}_6$}
\label{taba6}
\end{minipage}\hfill
\begin{minipage}[b]{0.6\linewidth}
\centering
  \begin{tikzpicture}[>=stealth',semithick,auto]
    \tikzstyle{subj} = [circle, minimum width=6pt, fill, inner sep=0pt]
    \tikzstyle{obj}  = [circle, minimum width=6pt, draw, inner sep=0pt]

    \tikzstyle{every label}=[font=\bfseries]

    \node[subj,  label=below:0] (0) at (0,0) {};
    \node[subj,  label=below:c] (c) at (-1,1) {};
    \node[subj,  label=below:a] (a) at (1,.5) {};
    \node[subj,  label=below right:b] (b) at (1,1.5) {};
    \node[subj,  label=below:d] (d) at (0,2) {};
    \node[subj,  label=below right:1] (1) at (0,3) {};

    \path[-]   (0)    edge                node{}      (a);
    \path[-]   (a)    edge                node{}      (b);
    \path[-]   (0)    edge                node{}      (c);
    \path[-]   (c)    edge                node{}      (d);
    \path[-]   (b)    edge                node{}      (d);
    \path[-]   (d)    edge                node{}      (1);
\end{tikzpicture}

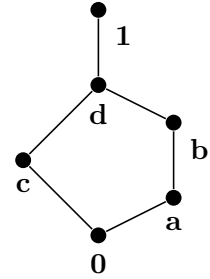
\captionof{figure}{Hasse diagram of $\mathfrak{A}_{6}$}
\label{figa6}
\end{minipage}
\end{table}
\FloatBarrier
\end{example}
\begin{example}\label{exa8}
Consider the lattice $A_8=\{0,a,b,c,d,e,f,1\}$ with the Hasse diagram shown in Figure \ref{figa8}. Routine calculations reveal that $\mathfrak{A}_8=(A_8;\vee,\wedge,\odot,\rightarrow,0,1)$ is a residuated lattice. The commutative operation $\odot$ is defined by Table \ref{taba8}, and the operation $\rightarrow$ is articulated by $x \rightarrow y = \bigvee \{z \in A_8 \mid x \odot z \leq y\}$, for any $x, y \in A_8$.
\FloatBarrier
\begin{table}[ht]
\begin{minipage}[b]{0.56\linewidth}
\centering
\begin{tabular}{ccccccccc}
\hline
$\odot$ & 0 & a & b & c & d & e & f & 1 \\ \hline
0       & 0 & 0 & 0 & 0 & 0 & 0 & 0 & 0 \\
        & a & 0 & a & a & a & a & a & a \\
        &   & b & 0 & 0 & 0 & 0 & b & b \\
        &   &   & c & c & a & c & a & c \\
        &   &   &   & d & a & a & d & d \\
        &   &   &   &   & e & c & d & e \\
        &   &   &   &   &   & f & f & f \\
        &   &   &   &   &   &   & 1 & 1 \\ \hline
\end{tabular}
\caption{Cayley table for ``$\odot$" of $\mathfrak{A}_8$}
\label{taba8}
\end{minipage}\hfill
\begin{minipage}[b]{0.6\linewidth}
\centering
  \begin{tikzpicture}[>=stealth',semithick,auto]
    \tikzstyle{subj} = [circle, minimum width=6pt, fill, inner sep=0pt]
    \tikzstyle{obj}  = [circle, minimum width=6pt, draw, inner sep=0pt]

    \tikzstyle{every label}=[font=\bfseries]

    \node[subj,  label=below:0] (0) at (0,0) {};
    \node[subj,  label=below:a] (a) at (-1,1) {};
    \node[subj,  label=below:b] (b) at (1,1) {};
    \node[subj,  label=below:c] (c) at (-2,2) {};
    \node[subj,  label=below:d] (d) at (0,2) {};
    \node[subj,  label=below:e] (e) at (-1,3) {};
    \node[subj,  label=below:f] (f) at (1,3) {};
    \node[subj,  label=below:1] (1) at (0,4) {};

    \path[-]   (0)    edge                node{}      (a);
    \path[-]   (0)    edge                node{}      (b);
    \path[-]   (b)    edge                node{}      (d);
    \path[-]   (d)    edge                node{}      (f);
    \path[-]   (f)    edge                node{}      (1);
    \path[-]   (a)    edge                node{}      (d);
    \path[-]   (a)    edge                node{}      (c);
    \path[-]   (c)    edge                node{}      (e);
    \path[-]   (d)    edge                node{}      (e);
    \path[-]   (e)    edge                node{}      (1);
\end{tikzpicture}

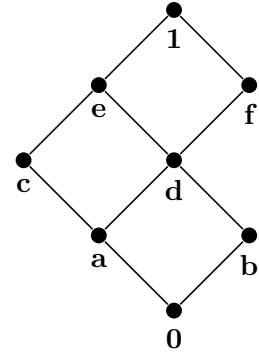
\captionof{figure}{Hasse diagram of $\mathfrak{A}_{8}$}
\label{figa8}
\end{minipage}
\end{table}
\FloatBarrier
\end{example}
\begin{example}\label{lukas}
Let $n$ be a fixed natural number. By \citet[p. 11, Example 2]{turunen1999mathematics} follows that $\L=([0,1];\max,\min,\odot,\rightarrow,0,1)$ is a residuated lattice, called \textit{generalized {\L}ukasiewicz structure}, in which the commutative operation $``\odot"$ is given by $x\odot y=(\max\{0,x^n+y^n-1\})^{\frac{1}{n}}$, and the operation $``\rightarrow"$ is given by $x\rightarrow y=\min\{1,(1-x^n+y^{n})^{\frac{1}{n}}\}$, for all $x,y\in [0,1]$.
\end{example}

Let $\mathfrak{A}$ be a residuated lattice. A non-void subset $F$ of $A$ is called a \textit{filter} of $\mathfrak{A}$ provided that $x,y\in F$ implies $x\odot y\in F$, and $x\vee y\in F$, for any $x\in F$ and $y\in A$. The set of filters of $\mathfrak{A}$ is denoted by $\mathscr{F}(\mathfrak{A})$. A filter $F$ of $\mathfrak{A}$ is called \textit{proper} if $F\neq A$. Clearly, $F$ is proper iff $0\notin F$. For any subset $X$ of $A$, the \textit{filter of $\mathfrak{A}$ generated by $X$} is denoted by $\mathscr{F}(X)$. For each $x\in A$, the filter generated by $\{x\}$ is denoted by $\mathscr{F}(x)$ and said to be \textit{principal}. The set of principal filters is denoted by $\mathscr{PF}(\mathfrak{A})$. Following \citet[\S 5.7]{gratzer2011lattice}, a join-complete lattice $\mathfrak{A}$, is called a \textit{frame} if it satisfies the join infinite distributive law (JID), i.e., for any $a\in A$ and $S\subseteq A$, $a\wedge \bigvee S=\bigvee \{a\wedge s\mid s\in S\}$.  According to \cite{galatos2007residuated}, $(\mathscr{F}(\mathfrak{A});\cap,\veebar,\textbf{1},A)$ is a frame in which $\veebar \mathcal{F}=\mathscr{F}(\cup \mathcal{F})$, for any $\mathcal{F}\subseteq \mathscr{F}(\mathfrak{A})$. With any filter $F$ of a residuated lattice $\mathfrak{A}$, we can associate a binary relation $\equiv_{F}$ on $A$ as follows; $x\equiv_{F} y$ if and only if $x\rightarrow y,y\rightarrow x\in F$, for any $x,y\in A$. The binary relation $\equiv_{F}$ is a congruence on $\mathfrak{A}$, and called \textit{the congruence induced by $F$ on $\mathfrak{A}$}. As usual, the set of all congruences on $\mathfrak{A}$ shall be denoted by $Con(\mathfrak{A})$. It is well-known that $\phi:\mathscr{F}(\mathfrak{A})\longrightarrow Con(\mathfrak{A})$, defined by $\phi(F)=\equiv_F$, is a lattice isomorphism, and this implies that $\mathbf{RL}$ is an ideal determined variety. For a filter $F$ of a residuated lattice $\mathfrak{A}$, the quotient set $A/\equiv_{F}$ with the natural operations becomes a residuated lattice which is denoted by $\mathfrak{A}/F$. Also, for any $a\in A$, the equivalence classes $a/\equiv_{F}$ is denoted by $a/F$.

\begin{example}\label{filterexa}
Consider the residuated lattice $\mathfrak{A}_4$ from Example \ref{exa4}, the residuated lattice $\mathfrak{A}_6$ from Example \ref{exa6}, and the residuated lattice $\mathfrak{A}_8$ from Example \ref{exa8}. The sets of their filters are presented in Table \ref{tafiex}.
\begin{table}[h]
\centering
\begin{tabular}{ccl}
\hline
                 & \multicolumn{2}{c}{Filters}                                       \\ \hline
$\mathfrak{A}_4$ & \multicolumn{2}{c}{$F_{1}=\{1\},F_{2}=\{a,1\},F_{3}=\{b,1\},F_{4}=A_4$} \\
$\mathfrak{A}_6$ & \multicolumn{2}{c}{$F_{1}=\{1\},F_{2}=\{a,b,d,1\},F_{3}=\{c,d,1\},F_{4}=\{d,1\},F_{5}=A_6$} \\
$\mathfrak{A}_8$ & \multicolumn{2}{c}{$F_{1}=\{1\},F_{2}=\{a,c,d,e,f,1\},F_{3}=\{c,e,1\},F_{4}=\{f,1\},F_{5}=A_8$} \\ \hline
\end{tabular}
\caption{The sets of filters of $\mathfrak{A}_4$, $\mathfrak{A}_6$, and $\mathfrak{A}_8$}
\label{tafiex}
\end{table}
\end{example}

Let $\mathfrak{A}$ and $\mathfrak{B}$ be residuated lattices. A map $f:A\longrightarrow B$ is called a  \textit{morphism of residuated lattices}, in symbols $f:\mathfrak{A}\longrightarrow \mathfrak{B}$, if it preserves the fundamental operations. If $f:\mathfrak{A}\longrightarrow \mathfrak{B}$ is a morphism of residuated lattices, we put $coker(f)=f^{\leftarrow}(1)$. One can see that $f$ is a injective if and only if $coker(f)=\{1\}$.
\begin{example}\label{mora6a4}
  Consider the residuated lattice $\mathfrak{A}_4$ from Example \ref{exa4}, the residuated lattice $\mathfrak{A}_6$ from Example \ref{exa6}. One can see that the map $f:A_{6}\longrightarrow A_{4}$, given by $f(0)=0$, $f(a)=f(b)=a$, $f(c)=b$ and $f(d)=f(1)=1$ is a morphism of residuated lattices.
\end{example}

A proper filter in a residuated lattice $\mathfrak{A}$ is termed \textit{maximal} if it is a maximal element within the set of all proper filters of $\mathfrak{A}$. The collection of maximal filters of $\mathfrak{A}$ is denoted by $Max(\mathfrak{A})$. Additionally, a proper filter $\mathfrak{p}$ in $\mathfrak{A}$ is \textit{prime} if $x\vee y\in \mathfrak{p}$ implies $x\in \mathfrak{p}$ or $y\in \mathfrak{p}$, for any $x, y \in A$. The set of prime filters of $\mathfrak{A}$ is denoted by $Spec(\mathfrak{A})$. Given the distributive lattice property of $\mathscr{F}(\mathfrak{A})$, it follows that $Max(\mathfrak{A})\subseteq Spec(\mathfrak{A})$. Zorn's lemma establishes that any proper filter is contained within a maximal filter and, consequently, within a prime filter. Minimal prime filters, identified as minimal elements within the set of prime filters, form the set $Min(\mathfrak{A})$. For a detailed exploration of prime filters in a residuated lattice, refer to \cite{rasouli2019going}.
\begin{example}\label{maxminex}
Consider the residuated lattice $\mathfrak{A}_4$ from Example \ref{exa4}, the residuated lattice $\mathfrak{A}_6$ from Example \ref{exa6}, and the residuated lattice $\mathfrak{A}_8$ from Example \ref{exa8}. The sets of their maximal, prime, and minimal prime filters are presented in Table \ref{prfiltab}.
\begin{table}[h]
\centering
\begin{tabular}{cccc}
\hline
                 & \multicolumn{3}{c}{Prime filters}      \\ \hline
                 & Maximal filters &          & Minimal prime filters      \\
$\mathfrak{A}_4$ &  $F_{2},F_{3}$                 &  &$F_{2},F_{3}$\\
$\mathfrak{A}_6$ &  $F_2,F_3$                 &  & $F_{1}$\\
$\mathfrak{A}_8$ &  $F_{2}$                          &         & $F_{3},F_{4}$ \\ \hline
\end{tabular}
\caption{The sets of maximal, prime, and minimal prime filters of $\mathfrak{A}_4$, $\mathfrak{A}_6$, and $\mathfrak{A}_8$}
\label{prfiltab}
\end{table}
\end{example}

Let $\mathfrak{A}$ be a residuated lattice, and $\Pi$ a collection of prime filters of $\mathfrak{A}$. For a subset $\pi$ of $\Pi$, we define $k(\pi)=\bigcap \pi$. Additionally, for a subset $X$ of $A$, we set $h_{\Pi}(X)=\{P\in \Pi \mid X\subseteq P\}$ and $d_{\Pi}(X)=\Pi\setminus h_{\Pi}(X)$.

The collection $\Pi$ can be topologized by considering the collection $\{h_{\Pi}(x)\mid x\in A\}$ as a closed (or open) basis. This topological structure is referred to as \textit{the (dual) hull-kernel topology} on $\Pi$ and is denoted by $\Pi_{h(d)}$. Moreover, the topology generated by $\tau_{h}\cup \tau_{d}$ on $\text{Spec}(\mathfrak{A})$ is known as \textit{the patch topology} and is denoted by $\tau_{p}$.

For a detailed discussion of the (dual) hull-kernel and patch topologies on a residuated lattice, we refer to \cite{rasouli2021hull}.

To clarify, we provide some examples of topological spaces associated with residuated lattices.
\begin{example}\label{spectrumofex}
Consider the residuated lattice $\mathfrak{A}_4$ from Example \ref{exa4}, the residuated lattice $\mathfrak{A}_6$ from Example \ref{exa6}, and the residuated lattice $\mathfrak{A}_8$ from Example \ref{exa8}. Table \ref{prfiltab} presents some topological spaces related to these structures.
\begin{table}[h]
\centering
\begin{tabular}{ccl}
\hline
                           & \multicolumn{2}{c}{Topological spaces}      \\ \hline
$Spec_{h}(\mathfrak{A}_4)$ & \multicolumn{2}{c} {$\{\emptyset,\{F_2\},\{F_3\},\{F_2,F_3\}\}$}\\
$Max_{d}(\mathfrak{A}_6)$ & \multicolumn{2}{c} {$\{\emptyset,\{F_{2}\},\{F_{3}\},\{F_2,F_3\}\}$}\\
$Min_{p}(\mathfrak{A}_8)$ & \multicolumn{2}{c} {$\{\emptyset,\{F_{3}\},\{F_{4}\},\{F_{3},F_{4}\},\{F_2,F_3,F_{4}\}\}$} \\ \hline
\end{tabular}
\caption{Some topological spaces associated with $\mathfrak{A}_4$, $\mathfrak{A}_6$, and $\mathfrak{A}_8$}
\label{prfiltab}
\end{table}
\end{example}

\section{\'{E}tal\'{e} spaces of residuated lattices over a topological space}\label{sec3}

This section introduces and investigates the notion of \'{e}tal\'{e}s of residuated lattices. We recall specific definitions of \'{e}tal\'{e}s of sets that will be used in our exploration. For the sake of completeness, we prove some properties of \'{e}tal\'{e}s of sets.
\begin{definition}\citep[\S 4.4]{engelking1989general}\label{localhomeodef}
A map $f:\mathscr{T}\longrightarrow \mathscr{B}$ is called a \textit{local homeomorphism} provided that for every point in $\mathscr{T}$ there exists a neighbourhood $U$ of this point such that the restriction $f|_{U}$ is a homeomorphism onto an open set.
\end{definition}
\begin{lemma}\label{lochomeocont}
A local homeomorphism is continuous.
\end{lemma}
\begin{proof}
Let $f:\mathscr{T}\longrightarrow \mathscr{B}$ be a local homeomorphism and $V$ an open subset of $\mathscr{B}$. Consider $t\in f^{\leftarrow}(V)$. There exists a neighborhood $U$ of $t$ such that $f|_{U}$ is a homeomorphism. Evidently, $(f|_{U})^{\leftarrow}(V\cap f(U))$ is a neighborhood of $t$ in $f^{\leftarrow}(V)$.
\end{proof}

\begin{lemma}\label{lochomeoope}
A local homeomorphism is open.
\end{lemma}
\begin{proof}
Let $f:\mathscr{T}\longrightarrow \mathscr{B}$ be a local homeomorphism. Assume that $W$ is an open subset of $\mathscr{T}$. Consider $b\in f(W)$. Pick $t\in f^{\leftarrow}(b)$. There exists a neighborhood $U$ of $t$ such that $f|_{U}$ is a homeomorphism. Evidently, $f|_{W}(U\cap W)$ is a neighborhood of $y$ in $f(U)$.
\end{proof}

Recall that if the domain of a map $f:\mathscr{T}\longrightarrow Y$ is a topological space, then the function $f$ is said to be \textit{locally injective} provided that for every point in
$\mathscr{T}$ there exists a neighbourhood $U$ of this point such that the restriction $f|_{U}$ is injective.
\begin{lemma}\label{lochomeolocinj}
A local homeomorphism is locally injective.
\end{lemma}
\begin{proof}
There is nothing to prove.
\end{proof}
\begin{theorem}\label{lhcol}
A map is a local homeomorphism if and only if it is continuous, open, and locally injective. In particular, a bijective local homeomorphism is a homeomorphism.
\end{theorem}
\begin{proof}
It is an immediate consequence of Lemmas \ref{lochomeocont}, \ref{lochomeoope} and \ref{lochomeolocinj}.
\end{proof}

\begin{proposition}\label{sheafbasepro}
Let $f:\mathscr{T}\longrightarrow \mathscr{B}$ be a local homeomorphism. The family $\{V\mid V~\textrm{is~open~in}~\mathscr{T}, f|_{V}~\textrm{is~a~homeomorphism}\}$ is a basis for the topology of $\mathscr{T}$.
\end{proposition}
\begin{proof}
Let $U$ be an open set in $\mathscr{T}$. Consider $x\in U$. There exists a neighborhood $W$ of $x$ such that $f|_{W}$ is a homeomorphism. Let $V=U\cap W$. It is easy to check that $f|_{V}$ is a homeomorphism and $x\in V\subseteq U$.
\end{proof}

\begin{lemma}\label{lochomeoprop}
\begin{enumerate}
\item [$(1)$ \namedlabel{lochomeoprop1}{$(1)$}] The restriction of a local homeomorphism to a topological subspace is a local homeomorphism;
\item [$(2)$ \namedlabel{lochomeoprop2}{$(2)$}] The composition of two local homeomorphisms is a local homeomorphism.
\end{enumerate}
\end{lemma}
\begin{proof}
It is evident.
\end{proof}

Let $\mathcal{C}$ be a category and $c$ an object of $\mathcal{C}$. Recall that the \textit{slice category} (an especial type of comma category \cite[\S 6 of Chapter II ]{lane1971categories}) of $\mathcal{C}$ over $c$, denoted by $\mathcal{C}/c$, is the category whose objects are all pairs $(a,f)$, where $a$ is an object of $\mathcal{C}$ and $f$ is a morphism from $a$ to $c$, and whose morphisms from $(a,f)$ to $(b,g)$ given by morphisms $h$ of $\mathcal{C}$ such that $gh=f$ (See Fig. \ref{fig:slicec}). If $\mathcal{C}/c$ is a slice category, $c$ is known as the \textit{base space} of the category $\mathcal{C}/c$.
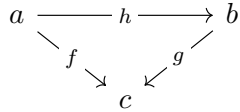
\begin{figure}[h!]
\centering
\begin{tikzcd}
a \arrow[rd, "f" description] \arrow[rr, "h" description] & & b \arrow[ld, "g" description] \\
& c &
\end{tikzcd}
\caption{Morphism $h$ of the slice category $\mathcal{C}/c$} \label{fig:slicec}
\end{figure}

Let $\textbf{Top}$ ($\textbf{Top}_{\ell}$) denote the category of topological spaces with continuous maps (local homeomorphisms) as morphisms. Given a topological space $\mathscr{B}$, the elements of the category $\mathbf{Bundle}(\mathscr{B})=\textbf{Top}/\mathscr{B}$ are referred to as \textit{bundles} over $\mathscr{B}$ (see \citet[\S 2.4]{maclane2012sheaves}), while the elements of the category $\mathbf{Etale}(\mathscr{B})=\textbf{Top}_{\ell}/\mathscr{B}$ are known as \textit{\'{e}tal\'{e}s} over $\mathscr{B}$. A morphism between bundles (\'{e}tal\'{e}s) over $\mathscr{B}$ is denoted as a \textit{morphism of bundles (\'{e}tal\'{e}s)} over $\mathscr{B}$.

If $(\mathscr{T},\pi)$ is an \'{e}tal\'{e} over $\mathscr{B}$, $\mathscr{T}$ is known as the \textit{total space}, and $\pi$ as the \textit{\'{e}tal\'{e} projection}. For a given $b\in \mathscr{B}$, the inverse image $\pi^{\leftarrow}(b)$ is referred to as the \textit{stalk of $\mathscr{T}$ at $b$} and denoted by $\mathscr{T}_{b}$. In the sequel, a bundle (\'{e}tal\'{e}) $(\mathscr{T},\pi)$ over $\mathscr{B}$ shall be denoted by $\mathscr{T}\overset{\pi}{\downarrow}\mathscr{B}$ and simply called a bundle (an \'{e}tal\'{e}). A morphism of bundles (\'{e}tal\'{e}s) $h:(\mathscr{T}_{1},\pi_1)\to (\mathscr{T}_{2},\pi_2)$ over $\mathscr{B}$ shall be denoted by $h:\mathscr{T}_{1}\overset{\pi_1}{\downarrow}\mathscr{B}\to \mathscr{T}_{2}\overset{\pi_2}{\downarrow}\mathscr{B}$ and simply called a morphism of bundles (\'{e}tal\'{e}s).

The following proposition, which is restated the lemma 3.5 of \cite[\S 2]{Tennison1976} with more detailed proof, establishes that the morphisms of bundles and \'{e}tal\'{e}s over a topological space are the same.
\begin{proposition}\label{conopnloc}
Let $\mathscr{T}\overset{\pi}{\downarrow} \mathscr{B}$ and $\mathscr{S}\overset{\phi}{\downarrow} \mathscr{B}$ be two \'{e}tal\'{e}s. If $h:\mathscr{T}\longrightarrow \mathscr{S}$ is a map such that $\phi h=\pi$, the following are equivalent:
\begin{enumerate}
\item [$(1)$ \namedlabel{conopnloc1}{$(1)$}] $h$ is a local homeomorphism;
\item [$(2)$ \namedlabel{conopnloc2}{$(2)$}] $h$ is open;
\item [$(3)$ \namedlabel{conopnloc3}{$(3)$}] $h$ is continuous.
\end{enumerate}
\end{proposition}
\begin{proof}
\item [\ref{conopnloc1}$\Rightarrow$\ref{conopnloc2}:] It is evident by Lemma \ref{lochomeoope}.
\item [\ref{conopnloc2}$\Rightarrow$\ref{conopnloc3}:] Applying Proposition \ref{sheafbasepro}, let $V$ be a basic open set in $\mathscr{S}$ such that $\phi|_{V}$ is a homeomorphism. Let $x\in h^{\leftarrow}(V)$. So $\pi(x)\in \phi(V)$. One an find a neighborhood $U$ of $x$ such that $\pi|_{U}$ is a homeomorphism and $\pi(U)\subseteq \phi(V)$. Let $W=h(U)\cap \phi|^{\leftarrow}_{V}(\pi(U))$. Consider $t\in \pi|^{\leftarrow}_{U}(\phi|_{V}(W))$. There exist $w\in W$ and $u\in U$ such that $\pi(u)=\phi(h(u))=\phi(w)=\pi(t)$, and this implies that $u=t$. This states that $t\in h^{\leftarrow}(V)$.
\item [\ref{conopnloc3}$\Rightarrow$\ref{conopnloc1}:] Let $t\in \mathscr{T}$. There exist neighborhoods $U$ and $V$, respectively, of $t$ and $h(t)$ such that $\pi|_{U}$ and $\phi|_{V}$ are homeomorphisms. By continuity of $h$, $W=U\cap h^{\leftarrow}(V)$ is a neighborhood of $t$. Let $h(t_{1})=h(t_{2})$, for some $t_{1},t_{2}\in W$. So we have $\pi(t_{1})=\phi(h(t_{1}))=\phi(h(t_{2}))=\pi(t_{2})$ and this states that $t_{1}=t_{2}$, since $\pi|_{U}$ is injective. Let $O$ be an open subset in $W$. So we have $h(O)=(\phi|_{V})^{\leftarrow}(\phi(h(O)))=(\phi|_{V})^{\leftarrow}(\pi(h(O)))$ and this shows that $h(O)$ is open.
\end{proof}

\begin{corollary}\label{etafulsubbun}
  Given a topological space $\mathscr{B}$, $\mathbf{Etale}(\mathscr{B})$ is a full subcategory of $\mathbf{Bundle}(\mathscr{B})$.
\end{corollary}
\begin{proof}
  Evident by Proposition \ref{conopnloc}.
\end{proof}

\begin{definition}
Let $\mathscr{T}\overset{\pi}{\downarrow} \mathscr{B}$ be an \'{e}tal\'{e}. A continuous map $\sigma:X\longrightarrow \mathscr{T}$, where $X$ is a subset of $\mathscr{B}$, is referred to as a \textit{section of $\mathscr{T}$ over $X$} if $\pi\sigma=1$. The set of all sections of $\mathscr{T}$ over $X$ is denoted by $\Gamma(X,\mathscr{T})$. Elements of $\Gamma(\mathscr{B},\mathscr{T})$ are termed \textit{global sections} of $\mathscr{T}$.
\end{definition}

\begin{lemma}\label{shepropo}
Let $\mathscr{T}\overset{\pi}{\downarrow} \mathscr{B}$ be an \'{e}tal\'{e}, and $t\in \mathscr{T}$. There exist a neighborhood $U$ of $\pi(t)$ and a section $\sigma$ of $\mathscr{T}$ over $U$ such that $\sigma\pi(t)=t$.
\end{lemma}
\begin{proof}
There exists a neighborhood $V$ of $t$ such that $\pi|_{V}$ is a homeomorphism. Take $U=\pi(V)$ and define $\sigma:U\longrightarrow \mathscr{T}$ by $\sigma(x)=(\pi|_{V})^{\leftarrow}(x)$, for all $x\in U$.
\end{proof}

\begin{lemma}\label{1shepropo}
Let $\mathscr{T}\overset{\pi}{\downarrow} \mathscr{B}$ be an \'{e}tal\'{e}, $\sigma$ a section of $\mathscr{T}$ over an open subset of $\mathscr{B}$, $x$ an element of the domain $\sigma$, and $V$ a neighborhood of $\sigma(x)$ such that $\pi|_{V}$ is a homeomorphism. There exists a neighborhood $U$ of $x$ such that $\sigma(U)\subseteq V$ and $\sigma|_{U} = (\pi|_{V} )^{\leftarrow}|_{U}$.
\end{lemma}
\begin{proof}
By hypothesis $U=\sigma^{\leftarrow}(V)$ is a neighborhood of $x$ in $\mathscr{B}$. Obviously, $\sigma(U)\subseteq V$. Also, we have $(\pi|_{V} )^{\leftarrow}|_{U}(x)=(\pi|_{V} )^{\leftarrow}|_{U}(\pi\sigma(x))=(\pi|_{V} )^{\leftarrow}(\pi\sigma(x))=\sigma(x)$, for any $x\in U$.
\end{proof}

Let's note that if $\mathscr{T}\overset{\pi}{\downarrow} \mathscr{B}$ is an \'{e}tal\'{e}, $\sigma:X\longrightarrow \mathscr{T}$ is a section with $\xymatrix{X\ar@{^(->}[r]^i & \mathscr B}$ and $\xymatrix{Z\ar@{^(->}[r]^k & X}$ is a subspace of $X$, then $\sigma k:Z\longrightarrow \mathscr{T}$ is a section. We sometimes write $\sigma|_Z$ for the composition $\sigma k$.

If $X, Y \subseteq \mathscr{B}$, and $\sigma \in \Gamma(X,\mathscr{T})$ and $\tau \in \Gamma(Y,\mathscr{T})$, the \textit{equalizer} of $\sigma|_{X\cap Y}$ and $\tau|_{X\cap Y}$ in the category $\textbf{Top}$ is denoted, for simplicity, by $\textsc{Eq}(\sigma,\tau)$. This equalizer is defined as the set $\{b \in X\cap Y \mid \sigma(b) = \tau(b)\}$ equipped with the subspace topology. By a slight abuse of language, we refer to it as the equalizer of $\sigma$ and $\tau$.
\begin{proposition}\label{sheprop}
Let $\mathscr{T}\overset{\pi}{\downarrow} \mathscr{B}$ be an \'{e}tal\'{e}. The following hold:
\begin{enumerate}
\item [$(1)$ \namedlabel{sheprop1}{$(1)$}] If $U$ and $V$ are open sets in $\mathscr{B}$, $\sigma\in \Gamma(U,\mathscr{T})$ and $\tau\in \Gamma(V,\mathscr{T})$, then $\textsc{Eq}(\sigma,\tau)$ is open;
\item [$(2)$ \namedlabel{sheprop2}{$(2)$}] if $\mathscr{T}$ is Hausdorff, $U$ and $V$ are open sets in $\mathscr{B}$, $\sigma\in \Gamma(U,\mathscr{T})$ and $\tau\in \Gamma(V,\mathscr{T})$, then $\textsc{Eq}(\sigma,\tau)$ is clopen;
\item [$(3)$ \namedlabel{sheprop3}{$(3)$}] Each section of $\mathscr{T}$, defined over an open set in $\mathscr{B}$, is open.
\end{enumerate}
\end{proposition}
\begin{proof}
\item [\ref{sheprop1}:] Let $b\in \textsc{Eq}(\sigma,\tau)$. Assume that $W$ is a neighborhood of $\sigma(b)$ such that $\pi|_{W}$ is a homeomorphism. Applying Lemma \ref{1shepropo}, there exist neighborhoods $U_{0}$ and $V_{0}$ of $b$ such that $\sigma(U_{0}),\tau(V_{0})\subseteq W$. One can see that $U_{0}\cap V_{0}$ is a neighborhood of $b$ contained in $\textsc{Eq}(\sigma,\tau)$.
\item [\ref{sheprop2}:] It follows by \ref{sheprop1} and \citet[Theorem 13.13]{willard2012general}.
\item [\ref{sheprop3}:] Let $U$ be an open set in $\mathscr{B}$ and $\sigma\in \Gamma(U,\mathscr{T})$. Consider $t\in \sigma(U)$. So $t=\sigma(b)$, for some $b\in U$. Let $W$ be a neighborhood $t$ such that $\pi|_{W}$ is a homeomorphism. By Lemma \ref{1shepropo}, there exists a neighborhood $V$ of $b$ such that $\sigma(V)\subseteq W$ and $\sigma|_{V} = (\pi|_{W} )^{\leftarrow}|_{V}$. So $\sigma(V)=(\pi|_{W} )^{\leftarrow}(V)$ is a neighborhood of $t$ contained in $\sigma(U)$.
\end{proof}

\begin{proposition}\label{1sheafbasepro}
Let $\mathscr{T}\overset{\pi}{\downarrow} \mathscr{B}$ be an \'{e}tal\'{e}. The family
$$\{\sigma(U)\mid U~\textrm{is~open~in}~\mathscr{B}, \sigma\in \Gamma(U,\mathscr{T})\}$$
is a basis for the topology of $\mathscr{T}$.
\end{proposition}
\begin{proof}
\sloppy{Let $V$ be an open set in $\mathscr{T}$ such that $\pi|_{V}$ is a homeomorphism. Take $U=\pi(V)$ and define $\sigma:U\longrightarrow \mathscr{T}$ by $\sigma(x)=(\pi|_{V})^{\leftarrow}(x)$, for all $x\in U$. It is evident that $\sigma\in \Gamma(U,\mathscr{T})$ and $V=\sigma(U)$. So $\{V\mid V~\textrm{is~Open~in}~\mathscr{T}, \pi|_{V}~\textrm{is~a~homeomorphism}\}\subseteq \{\sigma(U)\mid U~\textrm{is~Open~in}~\mathscr{B}, \sigma\in \Gamma(U,\mathscr{T})\}$ and so the result holds due to Proposition \ref{sheafbasepro}.}
\end{proof}

Let $X$ be a set, $\{Y_i\}_{i\in I}$ be a collection of topological spaces, and $\{f_i:Y_{i}\longrightarrow X\}_{i\in I}$ be a collection of maps. The \textit{final topology} induced by $\{f_i\}_{i\in I}$ on $X$ is the finest topology on $X$ which makes $f_i:Y_{i}\longrightarrow X$ continuous for all $i\in I$. It is obvious that a subset $U$ of $X$ is open in the final topology if and only if $f_{i}^{\leftarrow}(U)$ is open in $Y_{i}$, for each $i\in I$.
\begin{proposition}\label{finalshetopo}
Let $\mathscr{T}\overset{\pi}{\downarrow} \mathscr{B}$ be an \'{e}tal\'{e}. The following hold:
\begin{enumerate}
\item [$(1)$ \namedlabel{finalshetopo1}{$(1)$}] The topology of $\mathscr{T}$ coincides with the final topology induced by the set of sections of $\mathscr{T}$;
\item [$(2)$ \namedlabel{finalshetopo2}{$(2)$}] the topology induced on each stalk of $\mathscr{T}$ is the discrete topology.
\end{enumerate}
\end{proposition}
\begin{proof}
\item [\ref{finalshetopo1}:] Let $V$ be a subset of $\mathscr{T}$. If $V$ is open in the topology of $\mathscr{T}$, for every
section $\sigma$ of $\mathscr{T}$, $\sigma^{\leftarrow}(V)$ is an open subset of the domain $\sigma$. Thus $V$ is open in the final topology. Conversely, let $V$ be open in the final topology. Let $x\in V$. Applying Lemma \ref{shepropo}, there exists a section $\sigma$ such that $\sigma\pi(x)=x$. So by Proposition \ref{sheprop}\ref{sheprop3}, $\sigma\sigma^{\leftarrow}(V)$ is a neighborhood of $x$ contained in $V$. Thus $V$ is open in the topology of $\mathscr{T}$.
\item [\ref{finalshetopo2}:] Let $x\in \mathscr{B}$. By Lemma \ref{shepropo}, it follows that for any $y\in \mathscr{T}_{x}$ there exists a section $\sigma$ such that $\sigma\pi(y)=y$. One can see that $\sigma(U)\cap \mathscr{T}_{x}=\{y\}$.
\end{proof}
 \begin{definition} Let $f:X\to Y$ be a function.
 \begin{itemize}
 \item[\bf a)]
The \textit{kernel pair of $f$}, denoted as $\kappa_{f}$, is given by the following pullback:
\begin{center}
$\xymatrix{\kappa_{f} \ar[d]_{p_1} \ar[r]^{p_2} & X \ar[d]^f \\ X \ar[r]_f & Y}$
\end{center}
\item[\bf b)]
A map $\rho:\kappa_{f} \to X$ is said to be \textit{$f$-proper} provided that the following square commutes:
\begin{center}
$\xymatrix{\kappa_{f} \ar[d]_{p_1} \ar[r]^{\rho} & X \ar[d]^f \\ X \ar[r]_f & Y}$
\end{center}
\end{itemize}
\end{definition}

 \begin{proposition}\label{lemofkerpa}
Let $f:X\to Y$ be a function. With $X_{y}=f^{\leftarrow}(y)$ being the stalk at $y\in Y$, there is a one to one correspondence between,
\begin{itemize}
\item[$(1)$ \namedlabel{lemofkerpa1}{$(1)$}]
the collections of binary operations $\{\overset{f}{\diamond}_y:X_y\times X_y \to X_y\mid y\in Y\}$ and $f$-proper maps $\overset{f}{\diamond} : \kappa_{f} \to X$.
\item[$(2)$ \namedlabel{lemofkerpa2}{$(2)$}]
the collections of nullary operations $\{\overset{f}{\diamond}_y:\{y\} \to X_y\mid y\in Y\}$ and global sections $\overset{f}{\diamond} :Y \to X$.
\end{itemize}
\end{proposition}
\begin{proof}
\item [\ref{lemofkerpa1}:] Given $\{\overset{f}{\diamond}_y:X_y\times X_y \to X_y\mid y\in Y\}$, define $\overset{f}{\diamond}: \kappa_{f} \to X$ by $\overset{f}{\diamond}(x_1,x_2)=\overset{f}{\diamond}_{f(x_{1})}(x_{1},x_{2})$. It follows easily that $\overset{f}{\diamond}$ is $f$-proper. Conversely, given an $f$-proper map $\overset{f}{\diamond}:\kappa_{f} \to X$, define $\overset{f}{\diamond}_y(x_{1},x_{2})=\overset{f}{\diamond}(x_{1},x_{2})$. The commutativity of the square shows that $\overset{f}{\diamond}(x_{1},x_{2})\in X_y$. It can be easily verified that the correspondence is indeed one to one.
 \item [\ref{lemofkerpa2}:] Obvious.
\end{proof}
\begin{definition}
Let $f:X\to Y$ be a map. Given $f$-proper functions $\overset{f}{\vee},\overset{f}{\wedge},\overset{f}{\odot},\overset{f}{\to}:\kappa_{f}\to X$ and global sections $\overset{f}{0}, \overset{f}{1}:Y\to X$, we say
$(\kappa_{f};\overset{f}{\vee},\overset{f}{\wedge},\overset{f}{\odot},\overset{f}{\to},\overset{f}{0}, \overset{f}{1})$ is a \textit{stalk residuated lattice}, provided that $(X_y,\overset{f}{\vee}_y,\overset{f}{\wedge}_y,\overset{f}{\odot}_y,\overset{f}{\to}_y,\overset{f}{0}_y,\overset{f}{1}_y)$
is a residuated lattice, for any $y\in Y$.
\end{definition}

Next, we introduce the notion of bundle and \'{e}tal\'{e} of residuated lattices, inspired by the one obtained for universal algebras \citep{davey1973sheaf}.
 \begin{definition}\label{bund-etal of RL}
 A bundle (an \'{e}tal\'{e}) $\mathscr{T}\overset{\pi}{\downarrow}\mathscr{B}$ is said to be a bundle (an \'{e}tal\'{e}) of residuated lattices provided that the following assertions hold:
\begin{itemize}
\item[$(1)$ \namedlabel{etalresidef1}{$(1)$}] there are $\pi$-proper functions $\overset{\pi}{\vee},\overset{\pi}{\wedge},\overset{\pi}{\odot},\overset{\pi}{\to}$ and global sections $\overset{\pi}{0}, \overset{\pi}{1}$ making $\kappa_{\pi}$ a stalk residuated lattice;
 \item[$(2)$ \namedlabel{etalresidef2}{$(2)$}] the functions $\overset{\pi}{\vee},\overset{\pi}{\wedge},\overset{\pi}{\odot},\overset{\pi}{\to},\overset{\pi}{0}$ and $\overset{\pi}{1}$ are continuous.
\end{itemize}
\end{definition}
\begin{remark}
  The existence of the global section $\overset{f}{0}$ (or $\overset{f}{1}$) makes a bundle (an \'{e}tal\'{e}) of residuated lattices a \textit{retraction} in $\textbf{Top}$.
\end{remark}

To clarify the discussion, in the following, we will give some examples of \'{e}tal\'{e} of residuated lattices over a topological space.
\begin{example}\label{etspecha4}
Consider the residuated lattice $\mathfrak{A}_4$ from Example \ref{exa4} and the topological space $Spec_{h}(\mathfrak{A}_4)$ from Example \ref{spectrumofex}. Let $\mathscr{T}=\mathfrak{A}_{2}\coprod \mathfrak{A}_{2}=\{0_{1},1_{1},0_{2},1_{2}\}$ equipped with the discrete topology. Define $\pi:\mathscr{T}\to Spec(\mathfrak{A}_{4})$ by $\pi(0_{1})=\pi(1_{1})=F_{2}$ and $\pi(0_{2})=\pi(1_{2})=F_{3}$. One can see that $\mathscr{T}\overset{\pi}{\downarrow}Spec(\mathfrak{A}_{4})$ is an \'{e}tal\'{e} of residuated lattices.

\end{example}
\begin{example}\label{etmaxda6}
Consider the residuated lattice $\mathfrak{A}_6$ from Example \ref{exa6} and the topological space $Max_{d}(\mathfrak{A}_6)$ from Example \ref{spectrumofex}. Let $\mathscr{T}=\mathfrak{A}_{2}\coprod \mathfrak{A}_{3}=\{0_{1},1_{1},0_{2},a_{2},1_{2}\}$ equipped with the discrete topology. Define $\pi:\mathscr{T}\to Spec(\mathfrak{A}_{4})$ by $\pi(0_{1})=\pi(1_{1})=F_{2}$ and $\pi(0_{2})=\pi(a_{2})=\pi(1_{2})=F_{3}$. One can see that $\mathscr{T}\overset{\pi}{\downarrow}Max_{d}(\mathfrak{A}_6)$ is an \'{e}tal\'{e} of residuated lattices.
\end{example}
\begin{example}\label{etminpa8}
Consider the residuated lattice $\mathfrak{A}_8$ from Example \ref{exa8} and the topological space $Min_{p}(\mathfrak{A}_8)$ from Example \ref{spectrumofex}. Let $\mathscr{T}=\mathfrak{A}_{3}\coprod \mathfrak{A}_{4}=\{0_{1},a_{1},1_{1},0_{2},a_{2},b_{2},1_{2}\}$ equipped with the discrete topology. Define $\pi:\mathscr{T}\to Spec(\mathfrak{A}_{4})$ by $\pi(0_{1})=\pi(1_{1})=F_{2}$ and $\pi(0_{2})=\pi(a_{2})=\pi(1_{2})=F_{3}$. One can see that $\mathscr{T}\overset{\pi}{\downarrow}Min_{p}(\mathfrak{A}_8)$ is an \'{e}tal\'{e} of residuated lattices.
\end{example}

 \begin{proposition}\label{bunetalresipro}
Let $\mathscr{T}\overset{\pi}{\downarrow}\mathscr{B}$ be a bundle such that $\kappa_{f}$ is a stalk residuated lattice. Consider the following statements:
\begin{itemize}
\item [$(1)$ \namedlabel{bunetalresipro1}{$(1)$}] $\mathscr{T}\overset{\pi}{\downarrow}\mathscr{B}$ is a bundle of residuated lattices.
\item [$(2)$ \namedlabel{bunetalresipro2}{$(2)$}] $\Gamma(X,\mathscr{T})$ forms a residuated lattice under pointwise operations, for any subset $X$ of $\mathscr{B}$.
\item [$(3)$ \namedlabel{bunetalresipro3}{$(3)$}] $\Gamma(U,\mathscr{T})$ forms a residuated lattice under pointwise operations, for any open subset $U$ of $B$.
\end{itemize}
We have \ref{bunetalresipro1}$\Rightarrow$\ref{bunetalresipro2}$\Rightarrow$\ref{bunetalresipro3} and if $\mathscr{T}\overset{\pi}{\downarrow}\mathscr{B}$ is an \'{e}tal\'{e} of residuated lattices, then \ref{bunetalresipro3}$\Rightarrow$\ref{bunetalresipro1} as well.
\end{proposition}
\begin{proof}
\item [\ref{bunetalresipro1}$\Rightarrow$\ref{bunetalresipro2}:] Let $\overset{\pi}{\diamond}\in\{\overset{\pi}{\vee},\overset{\pi}{\wedge},\overset{\pi}{\odot},\overset{\pi}{\to}\}$. The corresponding map
$\diamond : \Gamma(X,\mathscr{T})\times\Gamma(X,\mathscr{T})\to \Gamma(X,\mathscr{T})$ is defined by $(\sigma \diamond \tau)(x)=\sigma(x)\overset{\pi}{\diamond}_x\tau(x)$, for any $x\in X$. We need to show for continuous sections $\sigma,\tau:X\to \mathscr{T}$, $\sigma \diamond \tau:X\to \mathscr{T}$ is continuous. Since $\pi\sigma=1=\pi\tau$, there is a unique map, that we denote by $\dblp{\sigma,\tau}$, rendering commutative the two triangles as shown below.
\begin{center}
$\xymatrix{X\ar@/_2pc/[rdd]_{\sigma} \ar[rd]|{\dblp{\sigma,\tau}} \ar@/^2pc/[rrd]^{\tau} & & \\ & \kappa_{\pi} \ar[d]_{p_1} \ar[r]^{p_2} & \mathscr{T} \ar[d]^{\pi} \\ & \mathscr{T} \ar[r]_{\pi} & \mathscr{B}}$
\end{center}
One can easily verify that $\sigma \diamond \tau=\overset{\pi}{\diamond}\circ \dblp{\sigma,\tau}$. Since $\dblp{\sigma,\tau}$ and $\overset{\pi}{\diamond}$ are continuous, so is their composition and thus $\sigma\diamond\tau$ is continuous.

 Now let $\overset{\pi}{\diamond}\in\{\overset{\pi}{0},\overset{\pi}{1}\}$. The corresponding section over $X$ is $\diamond:X\to \mathscr{T}$ defined by $\diamond = \overset{\pi}{\diamond}\circ i_X$, where $i_X:X\to \mathscr{B}$ is the inclusion.
Since $\overset{\pi}{\diamond}$ and $i_X$ are continuous, so is $\diamond$. It follows that $\Gamma(X,\mathscr{T})$ is a residuated lattice under the pointwise operations.
\item [\ref{bunetalresipro2}$\Rightarrow$\ref{bunetalresipro3}:] Obvious.
\item [\ref{bunetalresipro3}$\Rightarrow$\ref{bunetalresipro1}:] Let $\mathscr{T}\overset{\pi}{\downarrow}\mathscr{B}$ be an \'{e}tal\'{e} of residuated lattices. Using Proposition \ref{1sheafbasepro}, let $\sigma(U)$ be a basic open set of $\mathscr{T}$, for some open $U$ in $\mathscr{B}$ and section $\sigma$ in $\Gamma(U,\mathscr{T})$. Let $(x,y)\in \overset{\pi}{\diamond}^{\leftarrow}(\sigma(U))$. So $x\overset{\pi}{\diamond}_{b}y=\sigma(b)$, where $\pi(x)=b\in U$. Applying Lemma \ref{shepropo}, there exists a neighborhood $W$ of $b$ and sections $\sigma_{x},\sigma_{y}\in \Gamma(W,\mathscr{T})$ such that $\sigma_{x}(b)=x$ and $\sigma_{y}(b)=y$. By Proposition \ref{sheprop}\ref{sheprop1}, one can see that $V=\textsc{Eq}(\sigma,\sigma_{x}\diamond\sigma_{y})$ is a neighborhood of $b$. Since $(x,y)\in \sigma_{x}(V)\times \sigma_{y}(V)\cap \kappa_{\pi}\subseteq \overset{\pi}{\diamond}^{\leftarrow}(\sigma(U))$, so the result holds by Proposition \ref{sheprop}\ref{sheprop3}.
\end{proof}

\begin{definition}\label{resetalmor}
Let $\mathscr{T} \overset{\pi}{\downarrow}\mathscr{B}$ and $\mathscr{S}\overset{\phi}{\downarrow}\mathscr{B}$ be two bundles (\'{e}tal\'{e}s) of residuated lattices. A bundle (an \'{e}tal\'{e}) morphism $h:\mathscr{T} \overset{\pi}{\downarrow}\mathscr{B}\to \mathscr{S}\overset{\phi}{\downarrow}\mathscr{B}$ is said to be a \textit{morphism of bundles (\'{e}tal\'{e}s) of residuated lattices} provided that the following square commutes:
\[
\xymatrix{\kappa_{\phi} \ar[d]_{\dblp{hp_{1},hp_{2}}} \ar[r]^{\overset{\phi}{\diamond}} & \mathscr{T} \ar[d]^h \\ \kappa_{\psi} \ar[r]_{\overset{\psi}{\diamond}} & \mathscr{S}}
\]
\end{definition}
\begin{lemma}\label{etalresilattmorp}
Let $\mathscr{T} \overset{\pi}{\downarrow}\mathscr{B}$ and $\mathscr{S}\overset{\phi}{\downarrow}\mathscr{B}$ be two bundles (\'{e}tal\'{e}s) of residuated lattices and $h:\mathscr{T} \overset{\pi}{\downarrow}\mathscr{B}\to \mathscr{S}\overset{\phi}{\downarrow}\mathscr{B}$ a bundle (an \'{e}tal\'{e}) morphism. The following are equivalent:
\begin{enumerate}
\item [$(1)$ \namedlabel{etalresilattmorp1}{$(1)$}] $h$ is a morphism of bundles (\'{e}tal\'{e}s) of residuated lattices;
\item [$(2)$ \namedlabel{etalresilattmorp2}{$(2)$}] the restriction $h|_{\mathscr{T}_{b}}$, denoted as $h|_{b}$, is a morphism of residuated lattices, for each $b\in \mathscr{B}$.
\end{enumerate}
\end{lemma}
\begin{proof}
  It follows from the commutativity of the following square, for each $b\in \mathscr{B}$:
  \begin{center}
$\xymatrix{\mathscr{T}_{b}\times \mathscr{T}_{b}\ar[d]_{\dblp{hp_{1},hp_{2}}} \ar[r]^{\overset{\phi}{\diamond}_{b}} & \mathscr{T}_{b} \ar[d]^h \\ \mathscr{S}_{b}\times \mathscr{S}_{b} \ar[r]_{\overset{\psi}{\diamond}_{b}} & \mathscr{S}_{b}}$
\end{center}
\end{proof}

\begin{theorem}\label{bundetalcateres}
Let $\mathscr{B}$ be a topological space. The following hold:
\begin{itemize}
\item [$(1)$ \namedlabel{bundetalcateres1}{$(1)$}] The class of bundles of residuated lattices over $\mathscr{B}$ with the class of morphisms of bundles of residuated lattices forms a category, denoted as $\textbf{RL-Bundle}(\mathscr{B})$.
\item [$(2)$ \namedlabel{bundetalcateres2}{$(2)$}] The class of \'{e}tal\'{e}s of residuated lattices over $\mathscr{B}$ with the class of morphisms of \'{e}tal\'{e}s of residuated lattices forms a category, denoted as $\textbf{RL-Etale}(\mathscr{B})$.
\end{itemize}
\end{theorem}
\begin{proof}
It is straightforward.
\end{proof}

The next result, which has a routine verification, demonstrates that for a given topological space $\mathscr{B}$, the inclusion functor $i_{\mathscr{B}}:\textbf{Etale}(\mathscr{B})\to \textbf{Bundle}(\mathscr{B})$, as introduced in Corollary \ref{etafulsubbun}, can be lifted to residuated lattices.
\begin{corollary}\label{rletafulsubcatrlbund}
For a given topological space $\mathscr{B}$, $\textbf{RL-Etale}(\mathscr{B})$ is a full subcategory of $\textbf{RL-Bundle}(\mathscr{B})$.
\end{corollary}

\section{Interrelation between bundles and \'{e}tal\'{e}s of residuated lattices}\label{sec4}
This section investigates the interrelation between bundles and \'{e}tal\'{e}s of residuated lattices. In the preceding section, we demonstrated that given a topological space $\mathscr{B}$, $\textbf{RL-Etale}(\mathscr{B})$ constitutes a full subcategory of $\textbf{RL-Bundle}(\mathscr{B})$. While in this section, we refine this result, stating that $\textbf{RL-Etale}(\mathscr{B})$ is coreflective in $\textbf{RL-Bundle}(\mathscr{B})$.

Let $X$ and $Y$ be topological spaces. We recall \citep[p. 285]{munkres2014topology} that if $C$ is a compact subspace of $X$
and $U$ is an open subset of $Y$, the sets $S(C,U)=\{f\in C(X,Y)\mid  f(C)\subseteq U\}$ form a subbasis for a topology on $C(X,Y)$ that is called the \textit{compact-open topology}.

In the sequel, for topological spaces $X$ and $Y$, we endow $\mathcal{C}(X,Y)$ with the compact-open topology, and $X\times Y$ with the product topology.
\begin{proposition}\label{C(B,-) adj to B times -}
Let $\mathscr{B}$ be a locally compact Hausdorff space. The following functors are adjoints:
\begin{center}
$\xymatrix{\textbf{Top} \ar@<0.5ex>[rr]^{\mathcal{C}(\mathscr{B},-)} && \textbf{Top} \ar@<0.5ex>[ll]^{\mathscr{B}\times -}}$
\end{center}
\end{proposition}
\begin{proof}
$\mathcal{C}(\mathscr{B},-)$ sends a continuous $f:X\to Y$ to $\mathcal{C}(\mathscr{B},f)=f\circ -:\mathcal{C}(\mathscr{B},X)\to \mathcal{C}(\mathscr{B},Y)$, which is continuous because for a compact $C\subseteq B$ and an open $V\subseteq Y$, we have $\mathcal{C}(\mathscr{B},f)^{\leftarrow}(S(C,V))=S(C,f^{\leftarrow}(V))$. It then easily follows that $\mathcal{C}(\mathscr{B},-)$ is a functor. Also it is easily verified that
$\mathscr{B}\times -$ is a functor.

Now, we need to show that $\textbf{Top}(\mathscr{B}\times -,-)\cong \textbf{Top}(-,C(\mathscr{B},-))$, i.e., the two functors are naturally isomorphic. Let $\mathscr{T}$ and $\mathscr{X}$ be topological spaces. To each $h:\mathscr{B}\times \mathscr{X} \to \mathscr{T}$ we correspond
$\hat{h}:\mathscr{X}\to \mathcal{C}(\mathscr{B},\mathscr{T})$, defined by $\hat{h}(x)(b)=h(b,x)$, and to each $k:\mathscr{X}\to \mathcal{C}(\mathscr{B},\mathscr{T})$ we correspond $\bar{k}:\mathscr{B}\times \mathscr{X}\to \mathscr{T}$, defined by $\bar{k}(b,x)=k(x)(b)$. One can easily verify that $\bar{\hat{h}}=h$ and $\hat{\bar{k}}=k$. Also $h$ is continuous if and only if $\hat{h}$ is, see \citep[Theorem 46.11]{munkres2014topology}. This proves $\textbf{Top}(\mathscr{B}\times \mathscr{X},\mathscr{T})\cong \textbf{Top}(\mathscr{X},\mathcal{C}(\mathscr{B},\mathscr{T}))$. Naturality can be shown by some computations.
\end{proof}

In the sequel, let $U:\textbf{Top}\to \textbf{Set}$ and $D:\textbf{Set}\to \textbf{Top}$ be the forgetful and the discrete functors, respectively.
 \begin{corollary}\label{U C(B,-) adj to (B times -) D}
The following functors are adjoints:
\begin{center}
$\xymatrix{\textbf{Top} \ar@<0.5ex>[rr]^{U\circ \mathcal{C}(\mathscr{B},-)} && \textbf{Set} \ar@<0.5ex>[ll]^{(\mathscr{B}\times -)\circ D}}$
\end{center}
\end{corollary}
\begin{proof}
Follows from Proposition \ref{C(B,-) adj to B times -} and the fact that $U$ is a right adjoint of $D$. A direct proof shows that the assumption that $\mathscr{B}$ is a locally compact Hausdorff space is not needed.
\end{proof}

For a continuous function $\pi:\mathscr{T}\to \mathscr{B}$, putting the subspace topology on $\Gamma(\mathscr{B},\mathscr{T})\subseteq \mathcal{C}(\mathscr{B},\mathscr{T})$ and letting $\pi_1:\mathscr{B}\times \mathscr{X}\to \mathscr{B}$ be the natural projection, we have:
\begin{proposition}\label{h cor to dot h}
Let $\pi:\mathscr{T}\to \mathscr{B}$ be continuous and $\mathscr{X}$ a topological space. Suppose $h:\mathscr{B}\times \mathscr{X}\to \mathscr{T}$ is a continuous function and $\hat{h}:\mathscr{X}\to \mathcal{C}(\mathscr{B},\mathscr{T})$ its corresponding counterpart. Then in the following diagram, the left triangle commutes if and only if there is a map $\dot{h}:\mathscr{X}\to \Gamma(\mathscr{B},\mathscr{T})$ making the right triangle commute. Furthermore $\dot h$ is continuous.
\begin{center}
$\xymatrix{\mathscr{B}\times \mathscr{X} \ar[rr]^h \ar[rd]_{\pi_1} && \mathscr{T} \ar[ld]^{\pi}\\ & \mathscr{B} &}$ \hfil
$\xymatrix{\mathscr{X} \ar[rr]^{\hat{h}} \ar[rd]_{\dot{h}} && \mathcal{C}(\mathscr{B},\mathscr{T}) \\ & \Gamma(\mathscr{B},\mathscr{T}) \ar@{^(->}[ru] &}$
\end{center}
\end{proposition}
\begin{proof}
Suppose the left triangle commutes. For $x\in \mathscr{X}$, we have $(\pi \hat{h}(x))(b)=\pi(\hat{h}(x)(b))=\pi h(b,x)=\pi_1(b,x)=b$, so that $\pi \hat{h}(x)=1$. Thus $\hat{h}(x)\in \Gamma(\mathscr{B},\mathscr{T})$. Set $\dot{h}(x)=\hat{h}(x)$. Commutativity of the right triangle is obvious. Conversely, suppose there is a map $\dot{h}:\mathscr{X}\to \Gamma(\mathscr{B},\mathscr{T})$ making the right triangle commute. For $(b,x)\in \mathscr{B}\times \mathscr{X}$, we have $\pi h(b,x)=\pi \hat{h}(x)(b)=\pi\dot{h}(x)(b)=1(b)=\pi_1(b,x)$. Thus $\pi h=\pi_1$, as desired. Lastly the continuity of $\dot{h}$ follows from the commutativity of the right triangle.
\end{proof}

In the sequel, $\pi_{\mathscr{B}}=(\mathscr{B}\times -,\pi_1):\textbf{Top} \to \mathbf{Bundle}(\mathscr{B})$ is the functor that takes $\mathscr{X}$ to
$\mathscr{B}\times\mathscr{X}  \overset{\pi_1}{\downarrow}\mathscr{B}$.

\begin{theorem}\label{Gamma(B,-) adj to pi_1}
Let $\mathscr{B}$ be a locally compact Hausdorff space. Then the following functors are adjoints.
\begin{center}
$\xymatrix{\mathbf{Bundle}(\mathscr{B}) \ar@<0.5ex>[rr]^{\Gamma(\mathscr{B},-)} && \textbf{Top} \ar@<0.5ex>[ll]^{\pi_{\mathscr{B}}}}$
\end{center}
\end{theorem}
\begin{proof}
Follows from Propositions \ref{C(B,-) adj to B times -} and \ref{h cor to dot h}.
\end{proof}

\begin{corollary}\label{bundugset}
Let $\mathscr{B}$ be a topological space. The following are adjoint functors.
\begin{center}
$\xymatrix{\mathbf{Bundle}(\mathscr{B}) \ar@<0.5ex>[rr]^{U\circ\Gamma(\mathscr{B},-)} && \textbf{Set} \ar@<0.5ex>[ll]^{\pi_{\mathscr{B}}\circ D}}$
\end{center}
\end{corollary}
\begin{proof}
Follows from Theorem \ref{Gamma(B,-) adj to pi_1}. However a direct proof shows that the assumption that $\mathscr{B}$ is locally compact Hausdorff is not needed.
\end{proof}

\begin{theorem}\label{pi1 adj to U_B}
Let $\mathscr{B}$ be a locally compact Hausdorff space. The following functors are adjoints,
\begin{center}
$\xymatrix{\textbf{Top} \ar@<0.5ex>[rr]^{\pi_\mathscr{B}} && \mathbf{Bundle}(\mathscr{B}) \ar@<0.5ex>[ll]^{U_{\mathscr{B}}}}$
\end{center}
where $U_\mathscr{B}:\mathbf{Bundle}(\mathscr{B})\to \textbf{Top}$ is the forgetful functor.
\end{theorem}
\begin{proof}
Given a bundle $(X,f)$ over $\mathscr{B}$ and a topological space $Y$, $\textbf{Top}(U_{\mathscr{B}}(X,f),Y)\cong \mathbf{Bundle}(\mathscr{B})((X,f),\pi_{\mathscr{B}}(Y))$, as can be seen by the assignment that sends each continuous function $g:U_{\mathscr{B}}(X,f)\to Y$ to the continuous function $\langle f,g\rangle :(X,f)\to \pi_{\mathscr{B}}(Y)$ and each $k:(X,f)\to \pi_{\mathscr{B}}(Y)$ to
$\pi_{2} k$. The naturality follows easily.
\end{proof}

The adjunctions of \ref{C(B,-) adj to B times -}, \ref{Gamma(B,-) adj to pi_1} and \ref{pi1 adj to U_B} are related as follows.
\begin{theorem}\label{adjs related}
Let $\mathscr{B}$ be a locally compact Hausdorff space. In the following diagram, the upper and lower triangles commute.
\begin{center}
$\xymatrix{\textbf{Top} \ar@<0.5ex>[rr]^{\pi_{\mathscr{B}}} \ar@/^2.5pc/[rrrr]^{\mathcal{C}(\mathscr{B},-)} && \mathbf{Bundle}(\mathscr{B}) \ar@<0.5ex>[ll]^{U_{\mathscr{B}}}
\ar@<0.5ex>[rr]^{\Gamma(\mathscr{B},-)} && \textbf{Top} \ar@<0.5ex>[ll]^{\pi_{\mathscr{B}}} \ar@/^2.5pc/[llll]^{\mathscr{B}\times -}
}$
\end{center}
\end{theorem}

Let $\mathfrak{A}=(A;\vee,\wedge,\odot,\rightarrow,0,1)$ be a residuated lattice and $\tau$ be a topology on $A$. We recall \cite[Definition 3.4]{rasouli2020topological} that the pair $\mathfrak{A}_{\tau}=(\mathfrak{A};\tau)$ is said to be a \textit{topological residuated lattice} provided that each operation of $\mathfrak{A}$ is continuous, i.e., the fundamental operation $\sigma$ of $\mathfrak{A}$ is continuous as a function from $A\times A$ to $A$. The category of topological residuated lattices shall be denoted by $\textbf{RL-Top}$.
\begin{theorem}\label{C(B,-) fr RLTop to RLTop}
Let $\mathscr{B}$ be a topological space. The functor $\mathscr{C}(\mathscr{B},-):\textbf{Top} \to \textbf{Top}$ lifts to a functor $\mathscr{C}(\mathscr{B},-):\textbf{RL-Top} \to \textbf{RL-Top}$ in the sense the following square commutes:
\begin{center}
$\xymatrix{\textbf{RL-Top} \ar[d]_{U} \ar[rr]^{\mathscr{C}(\mathscr{B},-)} && \textbf{RL-Top} \ar[d]^{U} \\ \textbf{Top} \ar[rr]_{\mathscr{C}(\mathscr{B},-)} && \textbf{Top}
}$
\end{center}
where $U$ is the forgetful functor.
\end{theorem}
\begin{proof}
Let $\mathfrak{A}$ be a topological residuated lattice and $\sigma$ a binary fundamental operation. We define the operation $\overline{\sigma}$ on $\mathscr{C}(\mathscr{B},\mathfrak{A})$, pointwisely. To show $\overline{\sigma}$ is continuous, we observe that it is the composition of the two maps shown below.
\[\xymatrix{
\mathscr{C}(\mathscr{B},\mathfrak{A})\times \mathscr{C}(\mathscr{B},\mathfrak{A}) \ar@{->}[rd]_{\theta} \ar@{->}[rr]^{\overline{\sigma}} &  & \mathscr{C}(\mathscr{B},\mathfrak{A}) \\
 & \mathscr{C}(\mathscr{B},\mathfrak{A}\times \mathfrak{A}) \ar@{->}[ru]_{{\mathscr{C}(\mathscr{B},\sigma)}} &
}
\]
$\mathscr{C}(\mathscr{B},\sigma)$ is continuous by Proposition \ref{C(B,-) adj to B times -} and the fact that $\sigma$ is continuous. So we need to show that $\theta$ is continuous. Since the collection
\[\{S(C,U\times V) \mid C \hbox{ compact subset of } \mathscr{B} \hbox{ and } U, V \hbox{open subsets of } A\}\]
constitutes a subbasis for $\mathscr{C}(\mathscr{B},\mathfrak{A}\times \mathfrak{A})$ as well and
$\theta^{\leftarrow}(S(C,U\times V))=S(C,U)\times S(C,V)$, continuity of $\theta$ follows. This states that $\mathscr{C}(\mathscr{B},\mathfrak{A})$ is a topological residuated lattice. One can also verify that $\mathscr{C}(\mathscr{B},-)$ sends a continuous morphism $h:\mathfrak{A}\to \mathfrak{B}$ to a continuous morphism $\mathscr{C}(\mathscr{B},h):\mathscr{C}(\mathscr{B},\mathfrak{A})\to \mathscr{C}(\mathscr{B},\mathfrak{B})$. Preservation of identities and composition follows from that of $\mathscr{C}(\mathscr{B},-):\textbf{Top} \to \textbf{Top}$.
\end{proof}
\begin{theorem}
Let $\mathscr{B}$ be a topological space. The functor $U\circ \mathscr{C}(\mathscr{B},-):\textbf{Top} \to \textbf{Set}$ lifts to a functor $U\circ \mathscr{C}(\mathscr{B},-):\textbf{RL-Top} \to \textbf{RL}$ in the sense that the following squares commute:
\[
\xymatrix{\textbf{RL-Top} \ar[d]_{U} \ar[rr]^{\mathscr{C}(\mathscr{B},-)} && \textbf{RL-Top} \ar[d]|{U} \ar[r]^{U} & \textbf{RL} \ar[d]^{U} \\ \textbf{Top} \ar[rr]_{\mathscr{C}(\mathscr{B},-)} && \textbf{Top} \ar[r]_{U} & \textbf{Set}
}
\]
where the functors $U$ are the obvious forgetful functors.
\end{theorem}
\begin{proof}
Follows from Corollary \ref{U C(B,-) adj to (B times -) D}, Theorem \ref{C(B,-) fr RLTop to RLTop}, and the fact that the right square in the above diagram commutes.
\end{proof}
\begin{theorem}\label{Gamma and pi B lift}
Let $\mathscr{B}$ be a topological space. The functor
\begin{itemize}
\item [$(1)$ \namedlabel{Gamma and pi B lift1}{$(1)$}] $\Gamma(\mathscr{B},-):\textbf{Bundle}(\mathscr{B}) \to \textbf{Top}$ lifts to a functor $\Gamma(\mathscr{B},-):\textbf{RL-Bundle}(\mathscr{B}) \to \textbf{RL-Top}$ in the sense that the following square commutes:
\begin{center}
$\xymatrix{\textbf{RL-Bundle}(\mathscr{B}) \ar[d]_{U} \ar[rr]^{\Gamma(\mathscr{B},-)} && \textbf{RL-Top} \ar[d]^{U} \\ \textbf{Bundle}(\mathscr{B}) \ar[rr]_{\Gamma(\mathscr{B},-)} && \textbf{Top}
}$
\end{center}
\item [$(2)$ \namedlabel{Gamma and pi B lift2}{$(2)$}] $\pi_\mathscr{B}: \textbf{Top} \to \textbf{Bundle}(\mathscr{B})$ lifts to a functor $\pi_{\mathscr{B}}: \textbf{RL-Top} \to \textbf{RL-Bundle}(\mathscr{B})$ in the sense that the following square commutes:
\begin{center}
$\xymatrix{\textbf{RL-Top}(\mathscr{B}) \ar[d]_{U} \ar[rr]^{\pi_{\mathscr{B}}} && \textbf{RL-Bundle}(\mathscr{B}) \ar[d]^{U} \\ \textbf{Top} \ar[rr]_{\pi_{\mathscr{B}}} && \textbf{Bundle}(\mathscr{B})}$
\end{center}

\end{itemize}
\end{theorem}
\begin{proof}
\item [\ref{Gamma and pi B lift1}:] Given a bundle of residuated lattices $\mathscr{T}\overset{\phi}{\downarrow}\mathscr{B}$, by Proposition \ref{bunetalresipro}, $\Gamma(\mathscr{B},\mathscr{T})$ is a residuated lattice and by  Theorem \ref{Gamma(B,-) adj to pi_1}, it is a topological space. We need to show that the residuated lattice operations are continuous. For $\dot\diamond\in\{\dot\vee, \dot\wedge, \dot\odot, \dot\rightarrow\}$, by the definition of $\diamond : \Gamma(B,T)\times\Gamma(B,T)\to \Gamma(B,T)$ as given in the proof of Proposition \ref{bunetalresipro}, one can verify the commutativity of the following square:
\[
\xymatrix{
\Gamma(\mathscr{B},\mathscr{T})\times\Gamma(\mathscr{B},\mathscr{T}) \ar@{->}[rr]^{\diamond} \ar@{>->}[d]_{\dblp{-,-}} &  & \Gamma(\mathscr{B},\mathscr{T}) \ar@{_{(}->}[d]^{i} \\
\mathscr{C}(\mathscr{B},\kappa_{\pi}) \ar@{->}[rr]_{\mathscr{C}(\mathscr{B},\dot\diamond)} &  & \mathscr{C}(\mathscr{B},\kappa_{\pi})
}
\]
in which $\dblp{-,-}$ is injection and $i$ is the inclusion. Since $\dblp{-,-}$, $\mathscr{C}(\mathscr{B},\dot\diamond)$ and $i$ are continuous, and $i$ is the inclusion, the continuity of $\diamond$ follows.
For $\dot\diamond\in\{\dot 0, \dot 1\}$, the corresponding map $\diamond=\dot\diamond$ and is therefore continuous. Hence, $\Gamma(\mathscr{B},\mathscr{T})$ is a topological residuated lattice.

Next we show that a bundle residuated lattice morphism goes to a topological residuated lattice morphism. So let $h:\mathscr{S} \overset{\phi}{\downarrow} \mathscr{B} \to \mathscr{T} \overset{\psi}{\downarrow} \mathscr{B}$ be a morphism in $\textbf{RL-Bundle}(\mathscr{B})$. We have the morphism $\kappa_h=\dblp{hp_1,hp_2}$ rendering commutative the triangles in the following diagram:
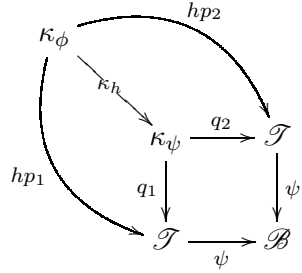
\[
\xymatrix{\kappa_{\phi} \ar@/_2pc/[rdd]_{hp_1} \ar[rd]|{\kappa_h} \ar@/^2pc/[rrd]^{hp_2} && \\
 & \kappa_{\psi} \ar[d]_{q_1} \ar[r]^{q_2} & \mathscr{T} \ar[d]^{\psi} \\
 & \mathscr{T} \ar[r]_{\psi} & \mathscr{B}}
 \]

With $\dot\diamond\in\{\dot\vee, \dot\wedge, \dot\odot, \dot\rightarrow\}$, the commutativity of the following square can be verified easily:
\begin{equation}\label{K-morphism}
\xymatrix{\kappa_{\phi} \ar[d]_{\dot\diamond} \ar[rr]^{\kappa_h} && \kappa_{\psi} \ar[d]^{\dot\diamond} \\ \mathscr{S} \ar[rr]_{h} && \mathscr{T}}
\end{equation}

Now consider the following diagram:
\[
\xymatrix{\Gamma(\mathscr{B},\mathscr{S})\times\Gamma(\mathscr{B},\mathscr{S}) \ar[d]_{\dblp{-,-}} \ar@/_4pc/[dd]_{\diamond} \ar[rrr]^{\mathscr{C}(\mathscr{B},h)\times \mathscr{C}(\mathscr{B},h)} &&& \Gamma(\mathscr{B},T)\times\Gamma(\mathscr{B},\mathscr{T}) \ar[d]^{\dblp{-,-}}
\ar@/^4pc/[dd]^{\diamond} \\
\mathscr{C}(\mathscr{B},\kappa_{\phi}) \ar[d]_{\mathscr{C}(\mathscr{B},\dot\diamond)} \ar[rrr]|{\mathscr{C}(\mathscr{B},\kappa_{h})} &&& \mathscr{C}(\mathscr{B},\kappa_{\psi}) \ar[d]^{\mathscr{C}(\mathscr{B},\dot\diamond)} \\
\Gamma(\mathscr{B},S) \ar[rrr]_{\mathscr{C}(\mathscr{B},h)} &&& \Gamma(\mathscr{B},T)}
\]

The left and right triangles commute by definition of $\diamond$. The bottom square commutes by using the commutativity of \eqref{K-morphism}. The commutativity of the top square follows from the definitions of $\dblp{-,-}$ and $\kappa_h$. Hence, the big square in the above diagram commutes, implying the preservation of $\diamond$ by $\Gamma(\mathscr{B},-)$.
For $\dot\diamond\in\{\dot 0, \dot 1\}$, since for all $b\in B$, $h_b\circ\diamond_b=\diamond_b$, we get $(h\circ -)(\diamond)=h\circ \diamond=\diamond$. Hence $\Gamma(\mathscr{B},h)=h\circ -$
is in $\textbf{RL-Top}$.

Preservation of identities and composition follows from the fact that $\Gamma(B,-):\textbf{Bundle}(\mathscr{B}) \to \textbf{Top}$ is a functor. Finally, commutativity of the square involving the forgetful functors is seen easily.
\item [\ref{Gamma and pi B lift2}:] Given a topological resituated lattice $\mathfrak{A}$, we need to show that $\mathscr{B}\times \mathfrak{A} \overset{\pi_1}{\downarrow} \mathscr{B}$ is a bundle of residuated lattices. For each $b\in B$, we define
$\diamond_b : (\{b\}\times A)\times (\{b\}\times A)\to \{b\}\times A$ by $(b,x)\diamond_b (b,y)=(b,x\diamond_\mathfrak{A} y)$, where $\diamond_{\mathfrak{A}}\in\{\vee_{\mathfrak{A}}, \wedge_{\mathfrak{A}}, \odot_{\mathfrak{A}}, \rightarrow_{\mathfrak{A}}\}$,
and we define $0_b=(b,0_{\mathfrak{A}})$ and $1_b=(b,1_{\mathfrak{A}})$. The kernel pair $\kappa_{\pi_1}$ is $\mathscr{B}\times \mathfrak{A}\times \mathfrak{A}$, and the corresponding map $\dot\diamond:\mathscr{B}\times \mathfrak{A}\times \mathfrak{A} \to \mathscr{B}\times \mathfrak{A}$ is $1_B\times \diamond_\mathfrak{A}$, which is continuous, because both $1_B$ and $\diamond_{\mathfrak{A}}$ are. Thus $\mathscr{B}\times \mathfrak{A} \overset{\pi_1}{\downarrow} \mathscr{B}$ is in $\textbf{RL-Bundle}(\mathscr{B})$.

Also a continuous map $f:X\to Y$ goes to $1_B\times f : B\times X \overset{\pi_1}{\downarrow} B \to B\times Y \overset{\pi_1}{\downarrow} B$, which can be easily verified to be a bundle residuated lattice morphism. What remains to show can be easily verified too.
\end{proof}

The next result shows that the adjoint functors of Corollary \ref{bundugset} both lift to residuated lattices.
 \begin{theorem}
Let $\mathscr{B}$ be a topological residuated lattice. The functors $U\circ \Gamma(\mathscr{B},-):\textbf{Bundle}(\mathscr{B}) \to \textbf{Set}$ and $\pi_{\mathscr{B}}\circ D : \textbf{Set} \to \textbf{Bundle}(\mathscr{B})$ both lift to residuated lattices, i.e, the following squares commute:
\begin{center}
$\xymatrix{\textbf{RL-Bundle}(\mathscr{B}) \ar[d]_{U} \ar@<0.5ex>[rr]^{\Gamma(\mathscr{B},-)} && \textbf{RL-Top} \ar@<0.5ex>[ll]^{\pi_\mathscr{B}} \ar[d]|{U} \ar[r]^{U} & \textbf{RL} \ar@<0.9ex>[l]^D \ar[d]^{U} \\ \textbf{Bundle}(\mathscr{B}) \ar@<0.5ex>[rr]^{\Gamma(\mathscr{B},-)} && \textbf{Top} \ar@<0.5ex>[ll]^{\pi_{\mathscr{B}}} \ar[r]^{U} & \textbf{Set} \ar@<0.9ex>[l]^D
}$
\end{center}
\end{theorem}
\begin{proof}
Follows from Corollary \ref{bundugset}, Theorem \ref{Gamma and pi B lift}, and the fact that the right squares in the above diagram commute with respect to both $U$ and $D$.
\end{proof}

Building upon Corollary \ref{etafulsubbun}, we have established that, for a given topological space $\mathscr{B}$, $\textbf{Etale}(\mathscr{B})$ forms a full subcategory of $\textbf{Bundle}(\mathscr{B})$. Now, we aim to improve this result by stating that $\textbf{Etale}(\mathscr{B})$ is coreflective in $\textbf{Bundle}(\mathscr{B})$. While \citet[\S 2.6, Corollary 3]{maclane2012sheaves} has stated this result using presheaves, we aim to establish it independently of them.
\begin{proposition}\label{etafulsubcatbund}
  Given a topological space $\mathscr{B}$, $\textbf{Etale}(\mathscr{B})$ is coreflective in $\textbf{Bundle}(\mathscr{B})$.
\end{proposition}
\begin{proof}
To prove this, we show the inclusion $i_{\mathscr{B}}:\textbf{Etale}(\mathscr{B})\to \textbf{Bundle}(\mathscr{B})$ has a right adjoint $R_{\mathscr{B}}:\textbf{Bundle}(\mathscr{B}) \to \textbf{Etale}(\mathscr{B})$. So let a bundle $X\overset{f}{\downarrow}\mathscr{B}$ be given. For each $b\in \mathscr{B}$, set $\Gamma(f)_b=\bigcup\limits_{U\in N_b} \Gamma(U,X)$, where $N_{b}$ is the set of neighbourhoods of $b$.
Define a relation $\sim_b$ on the set $\Gamma(f)_b$ as follows; for sections $s:U\to \mathscr{T}$ and $t:V\to \mathscr{t}$ in $\Gamma(f)_b$, $s\sim_{b} t$ if and only if there exists $W\in N_b$ contained in $U\cap V$, such that the following diagram commutes:
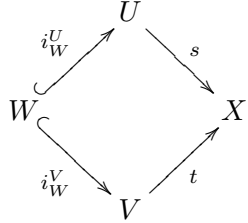
\begin{center}
$\xymatrix{ & U \ar[rd]^s & \\ W \ar@{^(->}[ru]^{i_W^U} \ar@{^(->}[rd]_{i_W^V} & & X \\ & V \ar[ru]_t
}$
\end{center}

 It is easy to verify that $\sim_b$ is an equivalence relation. For a given neighbourhood $U$ of $b$, the equivalence class of a section $s\in \Gamma(U,\mathscr{T})$ shall be denoted by $[s]_{U,b}$, or simply by $[s]_{b}$ if we are able to forget about the neighbourhood on which $s$ is defined. Set $\tilde{\Gamma}(f)_b=\Gamma(f)_b/\sim_b$ and $\tilde{\Gamma}(f)=\coprod\limits_{b\in B} \tilde{\Gamma}(f)_b$. For any open set $U$ of $\mathscr{B}$ and $s\in \Gamma(U,X)$, let the map $s_{U}:U\longrightarrow \tilde{\Gamma}(f)$ be given by $x\mapsto [s]_{x}$. We topologize $\tilde{\Gamma}(f)$ by putting on it the final topology coinduced by the set $\{s_U:U \hbox{ open in } \mathscr{B}, s\in \Gamma(U,X)\}$.

Let $U,V$ be open sets in $\mathscr{B}$, $s\in \Gamma(U,X)$ and $t\in \Gamma(V,X)$. With a little bit of effort one infers that $b\in t_{V}^{\leftarrow}(Im(s_{U}))$ if and only if there exists a neighborhood $W_{b}$ of $b$ contained in $U\cap V$ such that $s|_{W_{b}}=t|_{W_{b}}$. One can see that $t_{V}^{\leftarrow}(Im(s_{U}))=\bigcup_{b\in t_{V}^{\leftarrow}(Im(s_{U}))}W_{b}$. Now, let $O$ be an open set in the final topology. Consider $[s]_{U,x}\in O$. It is evident that $[s]_{U,x}=[s|_{s_{U}^\leftarrow(O)}]_{s_{U}^\leftarrow(O),x}\in Im((s|_{s_{U}^\leftarrow(O)})_{s_{U}^\leftarrow(O)})\subseteq O$. This verifies that the set $\{Im(s_{U})\mid U \hbox{ open in } \mathscr{B}, s\in \Gamma(U,X)\}$ forms a basis for the final topology on $\tilde{\Gamma}(f)$. Define $\tilde{f}:\tilde{\Gamma}(f)\longrightarrow \mathscr{B}$ to take $[s]_{b}$ to $b$. Continuity of $\tilde{f}$ follows by $\tilde{f}^{\leftarrow}(U)=\bigcup\{Im(s_{V})\mid V \hbox{ open subset of } U,~s\in \Gamma(V,X)\}$. Let $U$ be an open set in $\mathscr{B}$ and $s\in \Gamma(U,X)$. The openness of $\tilde{f}$ follows by $\tilde{f}(Im(s_{U}))=U$. Now, consider $[s]_{U,x}\in \tilde{\Gamma}(f)$. It is easy to see that $\tilde{f}|_{Im(s_{U})}$ is an injection and so $\tilde{f}$ is locally injective. Hence, by Theorem \ref{lhcol}, $\tilde{f}$ is a local homeomorphism, and so $\tilde{\Gamma}(f)\overset{\tilde{f}}{\downarrow}\mathscr{B}$ is an \'{e}tal\'{e}.

Now, let $h:X\overset{f}{\downarrow}\mathscr{B} \to Y\overset{g}{\downarrow}\mathscr{B}$ be a bundle morphism. Define the map
$\tilde{h}:\tilde{\Gamma}(f)\overset{\tilde{f}}{\downarrow}\mathscr{B} \to \tilde{\Gamma}(g)\overset{\tilde{g}}{\downarrow}\mathscr{B}$ by $\tilde{h}([s]_b)=[hs]_b$. Routine computations show that $\tilde{h}$ is well-defined and $\tilde{g}\tilde{h}=\tilde{f}$. Also, for a section $s:U\to X$,  the equality $\tilde{h}(Im(s_{U}))=Im((hs)_{U})$ yields openness of
$\tilde{h}$. Hence  by Proposition \ref{conopnloc}, $\tilde{h}$ is an \'{e}tal\'{e} morphism.

Now, define the mapping $R_\mathscr{B}$ to take the object $X \overset{f}{\downarrow} \mathscr{B}$ to $\tilde{\Gamma}(f) \overset{\tilde{f}}{\downarrow} \mathscr{B}$, and the morphism $h:X\overset{f}{\downarrow}\mathscr{B} \to Y\overset{g}{\downarrow}\mathscr{B}$ to $\tilde{h}:\tilde{\Gamma}(f)\overset{\tilde{f}}{\downarrow}\mathscr{B} \to \tilde{\Gamma}(g)\overset{\tilde{g}}{\downarrow}\mathscr{B}$. Preservation of identities and composition of $R_\mathscr{B}$ follows readily, and so $R_\mathscr{B}:\textbf{Bundle}(\mathscr{B})\to \textbf{Etale}(\mathscr{B})$ is a functor.

To show that $R_{\mathscr{B}}$ is a right adjoint of $i_{\mathscr{B}}$, we show that for each bundle $X \overset{f}{\downarrow} \mathscr{B}$, there is a couniversal morphism
$\varepsilon_f : \tilde{\Gamma}(f) \overset{\tilde{f}}{\downarrow} \mathscr{B} \to X \overset{f}{\downarrow} \mathscr{B}$. Define $\varepsilon_f$ by $\varepsilon_f([s]_b)=s(b)$. Routine computations show that $\varepsilon_f$ is well-defined and $f\varepsilon_f=\tilde{f}$. Injectivity can be verified by Proposition \ref{sheprop}\ref{sheprop1}, and openness follows by Proposition \ref{sheprop}\ref{sheprop3} and the fact that $\varepsilon_{f}(Im(s_{U}))=s(U)$. Continuity of $\epsilon_f$ follows by the fact that
$\epsilon_f^{-1}(O)=\bigcup s_U(s^{-1}(O))$, where the union runs over the set $\{U \hbox{ open subset of } \mathscr{B}, s\in \Gamma(U,X)\}$.

For couniversality, suppose a bundle morphism $h:\mathscr{T}\overset{\pi}{\downarrow} \mathscr{B} \to X \overset{f}{\downarrow} \mathscr{B}$ is given, with $\pi$ an \'{e}tal\'{e}. For a given $y\in \mathscr{T}$, by Lemma \ref{shepropo}, there exists a neighborhood $U$ of $\pi(y)$ and a section $s:U\to \mathscr{T}$ such that $s\pi(y)=y$. Define $\bar h:\mathscr{T}\overset{\pi}{\downarrow} \mathscr{B} \to \tilde{\Gamma}(f) \overset{\tilde f}{\downarrow} B$ by $\bar h(y)=[hs]_{\pi(y)}$. Some computations show that $\tilde{f}\bar{h}=\pi$ and that the following triangle commutes.
\begin{center}
$\xymatrix{\tilde{\Gamma}(f) \overset{\tilde{f}}{\downarrow} B \ar[rrr]^{\varepsilon_f} &&& X \overset{f}{\downarrow} B \\
\mathscr{T}\overset{\pi}{\downarrow} \mathscr{B} \ar[u]^{\bar h} \ar[rrru]_h
}$
\end{center}

Using injectivity of $\epsilon_f$, for any subset $O$ of $\mathscr{T}$, we have $\bar{h}^{-1}(O)=\bar{h}^{-1}\epsilon^{-1}_f\epsilon_f(O)=h^{-1}\epsilon_f(O)$. Thus continuity of
$\bar{h}$ follows by openness of $\varepsilon_f$ and continuity of $h$. Uniqueness of $\bar h$ follows from the fact that $\varepsilon_f$ is injective.
\end{proof}

So far we have shown that, given a topological space $\mathscr{B}$, $\textbf{RL-Etale}(\mathscr{B})$ is a full subcategory of $\textbf{RL-Bundle}(\mathscr{B})$ (Corollary \ref{rletafulsubcatrlbund}). Now we want to sharpen this result and say that $\textbf{RL-Etale}(\mathscr{B})$ is coreflective in $\textbf{RL-Bundle}(\mathscr{B})$.
\begin{theorem}\label{rletacorefrlbund}
  Given a topological space $\mathscr{B}$, $\textbf{RL-Etale}(\mathscr{B})$ is coreflective in $\textbf{RL-Bundle}(\mathscr{B})$.
\end{theorem}
\begin{proof}
We show the inclusion $i_{\mathscr{B}}:\textbf{RL-Etale}(\mathscr{B})\to \textbf{RL-Bundle}(\mathscr{B})$ has a right adjoint $R_{\mathscr{B}}:\textbf{RL-Bundle}(\mathscr{B}) \to \textbf{RL-Etale}(\mathscr{B})$. Let a bundle of residuated lattices $X\overset{f}{\downarrow}\mathscr{B}$ be given. By Proposition \ref{etafulsubcatbund}, we have an \'{e}tal\'{e}
$\tilde{\Gamma}(f) \overset{\tilde{f}}{\downarrow} \mathscr{B}$. To show this is an \'{e}tal\'{e} of residuated lattices, by Proposition \ref{bunetalresipro}, for each $U\in \mathscr{O(B)}$, $\Gamma(U,X)$ is a residuated lattice under the pointwise operations. Now, for each $b\in \mathscr{B}$, $U\in N_b$, and each operation $\diamond_U\in\{\vee_U, \wedge_U, \odot_U, \rightarrow_U\}$ on $\Gamma(U,X)$, define the operation $\tilde\diamond_b$ on $\tilde{\Gamma}(f)_b$, by $[s]_b\ \tilde\diamond_b\ [t]_b=[s_{|_{U\cap V}}]_{b}\ \diamond_{U\cap V}\ [t_{|_{U\cap V}}]_{b}$, where $s:U\to X$ and $t:V\to X$ are sections over $X$. Also, for the nullary operations $0_U$ and $1_U$ on $\Gamma(U,X)$, define $\tilde 0_b$ and $\tilde 1_b$ in $\tilde{\Gamma}(f)_b$ by $\tilde 0_b=[0_U]_b$ and $\tilde 1_b=[1_U]_b$. The corresponding $\tilde f$-proper operations make $\kappa_{\tilde f}$ a stalk residuated lattice.

To prove the continuity of the binary operations, consider the following commutative square,
\begin{center}
$\xymatrix{\kappa_{\tilde f} \ar[d]_{\dot{\tilde\diamond}}\ar[rr]^{\dblp{\varepsilon_f,\varepsilon_f}} && \kappa_f \ar[d]^{\dot\diamond} \\ \tilde{\Gamma}(f) \ar[rr]_{\varepsilon_f} && X
}$
\end{center}
where $\varepsilon_{f}$ is given in the proof of Proposition \ref{etafulsubcatbund}. The continuity of $\dot{\tilde\diamond}$ now follows from the injectivity and openness of $\varepsilon_f$ and the continuity of $\dot{\diamond}$ and $\dblp{\varepsilon_f,\varepsilon_f}$. The continuity of the sections $\tilde 0$ and $\tilde 1$ is easily verified.
This proves that $\tilde\Gamma(f) \overset{\tilde f}{\downarrow} \mathscr{B}$ is an \'{e}tal\'{e} of residuated lattices.

By Proposition \ref{etafulsubcatbund}, we know a bundle morphism $h: X \overset{f}{\downarrow} B \to Y \overset{g}{\downarrow} B$ goes to the \'{e}tal\'{e} morphism
$\tilde{h}: \tilde{\Gamma}(f) \overset{\tilde{f}}{\downarrow} B \to \tilde{\Gamma}(g) \overset{\tilde{g}}{\downarrow} B$, where $\tilde{h}([s]_b)=[hs]_b$.

To show $\tilde h$ preserves operations, we have:

\[
\begin{array}{ll}
   \tilde{h}([s]_b\tilde\diamond_b [t]_b)
   & = \tilde{h}([s_{|_{U\cap V}}]_{b}\ \diamond_{U\cap V}\ [t_{|_{U\cap V}}]_{b})\\
   & = \tilde{h}([s_{|_{U\cap V}}\diamond_b t_{|_{U\cap V}}]_{b}) \\
   & = [h(s_{|_{U\cap V}}\diamond_b t_{|_{U\cap V}})]_{b} \\
   & = [h(s_{|_{U\cap V}})\diamond_b h(t_{|_{U\cap V}})]_{b} \\
   & = [h s_{|_{U\cap V}}]_b\diamond_b [h t_{|_{U\cap V}}]_{b} \\
   & = \tilde{h}([s_{|_{U\cap V}}]_b)\diamond_b \tilde{h}([t_{|_{U\cap V}}]_{b}) \\
\end{array}
\]

The preservation of nullary operations can be easily verified. This proves that $\tilde{h}$ is a morphism of etales of residuated lattices. It follows that the functor $R_{\mathscr{B}}$ of $X \overset{f}{\downarrow} \mathscr{B}$ lifts to $R_{\mathscr{B}}:\textbf{RL-Bundle}(\mathscr{B}) \to \textbf{RL-Etale}(\mathscr{B})$.

To show that $R_{\mathscr{B}}$ is a right adjoint of $i_{\mathscr{B}}$, let $X \overset{f}{\downarrow} \mathscr{B}$ be a bundle of residuated lattices. We show that the morphism
$\varepsilon_f$ defined in Proposition \ref{etafulsubcatbund} preserves the operations and is thus a morphism in $\textbf{RL-Bundle}(\mathscr{B})$. We have:

\[
\begin{array}{ll}
   \varepsilon_f([s]_b\tilde{\diamond}_b [t]_b)
   & = \varepsilon_f([s_{|_{U\cap V}}]_{b}\ \diamond_{U\cap V}\ [t_{|_{U\cap V}}]_{b}) \\
   & = \varepsilon_f([s_{|_{U\cap V}}\diamond_b t_{|_{U\cap V}}]_{b}) \\
   & = (s_{|_{U\cap V}}\diamond_b t_{|_{U\cap V}})(b) \\
   & = s_{|_{U\cap V}}(b)\diamond_b t_{|_{U\cap V}}(b) \\
   & = \varepsilon_f([s]_b)\diamond_b \varepsilon_f([t]_b) \\
\end{array}
\]

Next we show $\varepsilon_f$ is couniversal. Given a bundle of residuated lattices morphism $h:\mathscr{T}\overset{\pi}{\downarrow} \mathscr{B} \to X \overset{f}{\downarrow} \mathscr{B}$, with $\pi$ an \'{e}tal\'{e} of residuated lattices, by Proposition \ref{etafulsubcatbund}, there is a unique etale-morphism $\bar h:\mathscr{T}\overset{\pi}{\downarrow} \mathscr{B} \to \tilde{\Gamma}(f) \overset{\tilde f}{\downarrow} B$ rendering commutative the following triangle.
\begin{center}
$\xymatrix{\tilde{\Gamma}(f) \overset{\tilde{f}}{\downarrow} B \ar[rrr]^{\varepsilon_f} &&& X \overset{f}{\downarrow} B \\
\mathscr{T}\overset{\pi}{\downarrow} \mathscr{B} \ar[u]^{\bar h} \ar[rrru]_h
}$
\end{center}

The commutativity of the above triangle, the preservation of the operations by $\varepsilon_f$ and $h$ and the fact that $\varepsilon_f$ is injective, imply that $\bar h$ preserves the operations. Thus $\bar h$ is a morphism in $\textbf{RL-Etale}(\mathscr{B})$ concluding the proof.
\end{proof}

\section{\'{E}tal\'{e}s of residuated lattices}\label{sec5}

So far, we have only been considering morphisms of \'{e}tal\'{e}s of residuated lattices defined over the same space. However, this section presents a method for transferring an \'{e}tal\'{e} of residuated lattices from one topological space to another, utilizing a continuous map. At the end, we define a contravariant functor, called the section functor, from the category of \'{e}tal\'{e}s of residuated lattices with inverse morphisms, denoted as $\textbf{RLE}_{inv}$, to the category of residuated lattices.

The next proposition shows that \'{e}tal\'{e}s over a topological space are pullback stable in $\textbf{Top}$.
\begin{proposition}\label{pb of es and em} Let $f:\mathscr B\to \mathscr C$ be a continuous map. 
\begin{enumerate}
 \item [$(1)$ \namedlabel{pb of es and em1}{$(1)$}] The pullback of an \'{e}tal\'{e} over $\mathscr{C}$ along $f$ is an \'{e}tal\'{e} over $\mathscr{B}$;
 \item [$(2)$ \namedlabel{pb of es and em2}{$(2)$}] the pullback of an \'{e}tal\'{e} morphism over $\mathscr{C}$ along $f$ is an \'{e}tal\'{e} morphism over $\mathscr{B}$.
\end{enumerate}
\end{proposition}
\begin{proof}
\item [\ref{pb of es and em1}:] Let $\mathscr{S}\overset{\phi}{\downarrow} \mathscr{C}$ be an \'{e}tal\'{e}. Consider the following pullback square in the category $\textbf{Top}$.
\begin{figure}[h!]
\centering
\begin{tikzcd}
f^*\mathscr{S} \arrow[dd, "f^*\phi" description] \arrow[rr, "f'" description] &  & \mathscr{S} \arrow[dd, "\phi" description] \\
                                                                          &  &                                            \\
\mathscr{B} \arrow[rr, "f" description]                                   &  & \mathscr{C}
\end{tikzcd}
\captionsetup{labelformat=empty}
\end{figure}

  We know that $f^*\mathscr{S}$ is the set $\{(b,s)\in \mathscr{B}\times \mathscr{S}\mid f(b)=\phi(s)\}$ endowed with the subspace topology induced by the product topology on $\mathscr{B}\times \mathscr{S}$, and that  $f^*\phi$ and $f'$ are natural projections. Given $(b,s)\in f^*\mathscr{S}$, let $U$ be an open neighborhood of $s$ such that $\phi|_U$ is a homeomorphism onto the open set $\phi(U)$. Thus $f'^{\leftarrow}(U)$ is an open neighborhood of $(b,s)$ and $\phi'(f'^{\leftarrow}(U))=f^{\leftarrow}(\phi(U))$ is open. One can easily verify that $f^*\phi|_{f'^{\leftarrow}(U)}$ is a homeomorphism.
\item [\ref{pb of es and em2}:]
 Let $h:\mathscr{S}_{1}\overset{\phi_{1}}{\downarrow} \mathscr{C}\to  \mathscr{S}_{2}\overset{\phi_{2}}{\downarrow} \mathscr{C}$ be a morphism of \'{e}tal\'{e}s. To get the pullback $f^*h$ of $h$ along $f$, we take the pullbacks of $\phi_{1}$ and $\phi_{2}$ along $f$ to get the bottom and big squares in the following diagram. Using the equality $ \phi_{2} h=\phi_{1}$, we get $f^*h$ as the unique morphism rendering commutative the top square and the left triangle.
 \begin{figure}[h!]
\centering
\begin{tikzcd}
f^*\mathscr{S}_{1} \arrow[rr, "f_{1}" description] \arrow[dd, "f^*h" description] \arrow[dddd, "f^*\phi_{1}" description, bend right=49] &  & \mathscr{S}_{1} \arrow[dd, "h" description] \arrow[dddd, "\phi_{1}" description, bend left=49] \\
                                                                                                                                         &  &                                                                                                \\
f^*\mathscr{S}_{2} \arrow[dd, "f^*\phi_{2}" description] \arrow[rr, "f_{2}" description]                                                 &  & \mathscr{S}_{2} \arrow[dd, "\phi_{2}" description]                                             \\
                                                                                                                                         &  &                                                                                                \\
\mathscr{B} \arrow[rr, "f" description]                                                                                                  &  & \mathscr{C}
\end{tikzcd}
\captionsetup{labelformat=empty}
\end{figure}

The commutativity of the left triangle makes $f^*h:f^*\mathscr{S}_{1}\overset{f^*\phi_{1}}{\downarrow}\mathscr{B}\to f^*\mathscr{S}_{2}\overset{f^*\phi_{2}}{\downarrow}\mathscr{B}$ a morphism of \'{e}tal\'{e}s.
\end{proof}
\begin{definition}
  Let $\mathscr{S}\overset{\phi}{\downarrow} \mathscr{C}$ be an \'{e}tal\'{e} and $f:\mathscr{B}\longrightarrow \mathscr{C}$ a continuous map. The \'{e}tal\'{e} $f^*\mathscr{S}\overset{f^*\phi}{\downarrow}\mathscr{B}$ is called
\text\it{the inverse image of} $\mathscr{S}\overset{\phi}{\downarrow} \mathscr{C}$ along $f$.
\end{definition}
\begin{theorem}\label{pb func on es}
Let $f:\mathscr{B}\longrightarrow \mathscr{C}$ be a continuous map. Pulling back along $f$ yields a functor $f^*:\mathbf{Etale}(\mathscr{C}) \longrightarrow \mathbf{Etale}(\mathscr{B})$.
\end{theorem}
\begin{proof}
Follows from Proposition \ref{pb of es and em} and \citet[p. 35]{johnstone1977topos}.
\end{proof}
\begin{proposition}\label{pb of rles and rles mor} Let $f:\mathscr B\to \mathscr C$ be a continuous map. 
\begin{enumerate}
 \item [$(1)$ \namedlabel{pb of rles and rles mor1}{$(1)$}] The pullback of an \'{e}tal\'{e} of residuated lattices over $\mathscr{C}$ along $f$ is an \'{e}tal\'{e} of residuated lattices over $\mathscr{B}$;
 \item [$(2)$ \namedlabel{pb of rles and rles mor2}{$(2)$}] the pullback of an \'{e}tal\'{e} of residuated lattices morphism over $\mathscr{C}$ along $f$ is an \'{e}tal\'{e} of residuated lattices morphism over $\mathscr{B}$.
\end{enumerate}
\end{proposition}
\begin{proof}
\item [\ref{pb of rles and rles mor1}:] Let $\mathscr{S}\overset{\phi}{\downarrow} \mathscr{C}$ be an \'{e}tal\'{e} of residuated lattices. By Proposition \ref{pb of es and em}\ref{pb of es and em1}, $f^*\mathscr{S}\overset{f^*\phi}{\downarrow}\mathscr{B}$ is an \'{e}tal\'{e}. We have the following pullbacks:

\[
\xymatrix{
f^{*}\mathscr{S} \ar@{->}[r]^{f'} \ar@{->}[d]_{f^{*}\phi} & \mathscr{S} \ar@{->}[d]^{\phi} & \kappa_{f^{*}\phi} \ar@{->}[r]^{q_{2}} \ar@{->}[d]_{q_{1}} & f^{*}\mathscr{S} \ar@{->}[d]^{f^{*}\phi} \\
\mathscr{B} \ar@{->}[r]_{f} & \mathscr{C} & f^{*}\mathscr{S} \ar@{->}[r]_{f^{*}\phi} & \mathscr{C} & }
\]

The commutativity of the above pullbacks yields the commutativity of the outer square in the following diagram. Since the inner square is a pullback, there is a unique morphism $\dblp{f'q_{1},f'q_{2}}$ making the triangles commute.

\[
\xymatrix{\kappa_{f^{*}\phi}\ar@/_2pc/[rdd]_{f'q_{1}} \ar[rd]|{\dblp{f'q_{1},f'q_{2}}} \ar@/^2pc/[rrd]^{f'q_{2}} & & \\ & \kappa_{\phi} \ar[d]_{p_1} \ar[r]^{p_2} & \mathscr{S} \ar[d]^{\phi} \\ & \mathscr{S} \ar[r]_{\phi} & \mathscr{C}}
\]

Let $\diamond$ be a $\phi$-proper map. So, there is a unique map $\overset{f^{*}\phi}{\diamond}$, making the triangles in the following diagram commute.
\[
\xymatrix{\kappa_{f^{*}\phi}\ar@/_2pc/[rdd]_{f^{*}q_{1}} \ar[rd]|{\overset{f^{*}\phi}{\diamond}} \ar@/^2pc/[rrd]^{\diamond\dblp{f'q_{1},f'q_{2}}} & & \\ & f^{*}\mathscr{S} \ar[d]_{f^{*}\phi} \ar[r]^{f'} & \mathscr{S} \ar[d]^{\phi} \\ & \mathscr{B} \ar[r]_{f} & \mathscr{C}}
\]

The commutativity of the lower triangle shows that $\overset{f^{*}\phi}{\diamond}$ is $f^{*}\phi$-proper and the commutativity of the upper triangle give the following commutative square.
\[
\xymatrix{\kappa_{f^{*}\phi}\ar[d]_{\dblp{f'q_{1},f'q_{2}}} \ar[r]^{\overset{f^{*}\phi}{\diamond}} & f^{*}\mathscr{S} \ar[d]^{f'} \\ \kappa_{\phi} \ar[r]_{\diamond} & \mathscr{S}}
\]

Now, the continuity of $\dblp{f'q_{1},f'q_{2}}$, $\diamond$ and $f'$, and the fact that the topology on $f^{*}\mathscr{S}$ is the induced topology by
$\xymatrix{f^{*}\mathscr{S} \ar[rr]^{(f^{*}\phi,f')} && \mathscr{B}\times \mathscr{S}}$, proves the continuity of $\overset{f^{*}\phi}{\diamond}$.
It follows that $f^*\mathscr{S}\overset{f^*\phi}{\downarrow}\mathscr{B}$ is an \'{e}tal\'{e} of residuated lattices.

\item [\ref{pb of rles and rles mor2}:] Let $h:\mathscr{S}_{1}\overset{\phi_{1}}{\downarrow} \mathscr{C}\to  \mathscr{S}_{2}\overset{\phi_{2}}{\downarrow} \mathscr{C}$ be a morphism of \'{e}tal\'{e}s of residuated lattices. By Proposition \ref{pb of es and em}\ref{pb of es and em2}, $f^*h:f^*\mathscr{S}_{1}\overset{f^*\phi_{1}}{\downarrow}\mathscr{B}\to f^*\mathscr{S}_{2}\overset{f^*\phi_{2}}{\downarrow}\mathscr{B}$ is a morphism of \'{e}tal\'{e}s. Using the pullbacks,
    \[
    \xymatrix{
\kappa_{\phi_{1}} \ar@{->}[r]^{p_{2}} \ar@{->}[d]_{p_{1}} & \mathscr{S}_{1} \ar@{->}[d]^{\phi_{1}} & \kappa_{\phi_{2}} \ar@{->}[r]^{q_{2}} \ar@{->}[d]_{q_{1}} & \mathscr{S}_{2} \ar@{->}[d]^{\phi_{2}} &  & f^{*}\mathscr{S}_{1} \ar@{->}[r]^{f_{1}} \ar@{->}[d]_{f^{*}h} \ar@/_3pc/@{->}[dd]_{f^*\phi_{1}} & \mathscr{S}_{1} \ar@{->}[d]^{h} \ar@/^3pc/@{->}[dd]^{\phi_{1}} \\
\mathscr{S}_{1} \ar@{->}[r]_{\phi_{1}} & \mathscr{C} & \mathscr{S}_{2} \ar@{->}[r]_{\phi_{2}} & \mathscr{C} &  & f^{*}\mathscr{S}_{2} \ar@{->}[r]^{f_{2}} \ar@{->}[d]_{f^{*}\phi_{2}} & \mathscr{S}_{2} \ar@{->}[d]^{\phi_{2}} \\
 &  &  &  &  & \mathscr{B} \ar@{->}[r]^{f} & \mathscr{C}
}
    \]

\[
\xymatrix{\kappa_{f^*{\phi_1}} \ar[d]_{p^*_1} \ar[r]^{p^*_2} & f^*\mathscr{S}_1 \ar[d]^{f^*\phi_1} \\ f^*\mathscr{S}_1 \ar[r]_{f^*\phi_1} & \mathscr{B}}
\hspace{1cm}
\xymatrix{\kappa_{f^*{\phi_2}} \ar[d]_{q^*_1} \ar[r]^{q^*_2} & f^*\mathscr{S}_2 \ar[d]^{f^*\phi_2} \\ f^*\mathscr{S}_2 \ar[r]_{f^*\phi_2} & \mathscr{B}}
\]

we get the unique morphisms making the triangles commute in the following diagrams.

\[
\xymatrix{
\kappa_{\phi_{1}} \ar@{->}[rd]^{{\dblp{hp_{1},hp_{2}}}} \ar@/^2pc/@{->}[rrd]^{hp_{2}} \ar@/_2pc/@{->}[rdd]_{hp_{1}} &  &  & \kappa_{f^{*}\phi_{1}} \ar@{->}[rd]^{{\dblp{f^{*}hp^{*}_{1},f^{*}hp^{*}_{2}}}} \ar@/^2pc/@{->}[rrd]^{f^{*}hp^{*}_{2}} \ar@/_2pc/@{->}[rdd]_{f^{*}hp^{*}_{1}} &  &  \\
 & \kappa_{\phi_{2}} \ar@{->}[r]^{q_{2}} \ar@{->}[d]_{q_{1}} & \mathscr{S}_{2} \ar@{->}[d]^{\phi_{2}} &  & \kappa_{f^{*}\phi_{2}} \ar@{->}[r]^{q^{*}_{2}} \ar@{->}[d]_{q^{*}_{1}} & f^{*}\mathscr{S}_{2} \ar@{->}[d]^{f^{*}\phi_{2}} \\
 & \mathscr{S}_{2} \ar@{->}[r]_{\phi_{2}} & \mathscr{C} &  & f^{*}\mathscr{S}_{2} \ar@{->}[r]_{f^{*}\phi_{2}} & \mathscr{B} \\
\kappa_{f^{*}\phi_{1}} \ar@{->}[rd]^{{\dblp{f_{1}p^{*}_{1},f_{1}p^{*}_{2}}}} \ar@/^2pc/@{->}[rrd]^{f_{1}p^{*}_{2}} \ar@/_2pc/@{->}[rdd]_{f_{1}p^{*}_{1}} &  &  & \kappa_{f^{*}\phi_{1}} \ar@{->}[rd]^{\overset{f^{*}\phi_{1}}{\diamond}} \ar@/^2pc/@{->}[rrd]^{{\overset{\phi_{1}}{\diamond}\dblp{f_{1}p^{*}_{1},f_{1}p^{*}_{2}}}} \ar@/_2pc/@{->}[rdd]_{f^{*}\phi_{1}p^{*}_{1}} &  &  \\
 & \kappa_{\phi_{1}} \ar@{->}[r]^{p_{2}} \ar@{->}[d]_{p_{1}} & \mathscr{S}_{1} \ar@{->}[d]^{\phi_{1}} &  & f^{*}\mathscr{S}_{1} \ar@{->}[r]^{f_{1}} \ar@{->}[d]_{f^{*}\phi_{1}} & \mathscr{S}_{1} \ar@{->}[d]^{\phi_{1}} \\
 & \mathscr{S}_{1} \ar@{->}[r]_{\phi_{1}} & \mathscr{C} &  & \mathscr{B} \ar@{->}[r]_{f} & \mathscr{C} \\
\kappa_{f^{*}\phi_{2}} \ar@{->}[rd]^{{\dblp{f_{2}q^{*}_{1},f_{2}q^{*}_{2}}}} \ar@/^2pc/@{->}[rrd]^{f_{2}q^{*}_{2}} \ar@/_2pc/@{->}[rdd]_{f_{2}q^{*}_{1}} &  &  & \kappa_{f^{*}\phi_{2}} \ar@{->}[rd]^{\overset{f^{*}\phi_{2}}{\diamond}} \ar@/^2pc/@{->}[rrd]^{{\overset{\phi_{2}}{\diamond}\dblp{f_{2}q^{*}_{1},f_{2}q^{*}_{2}}}} \ar@/_2pc/@{->}[rdd]_{f^{*}\phi_{2}q^{*}_{1}} &  &  \\
 & \kappa_{\phi_{2}} \ar@{->}[r]^{q_{2}} \ar@{->}[d]_{q_{1}} & \mathscr{S}_{2} \ar@{->}[d]^{\phi_{2}} &  & f^{*}\mathscr{S}_{2} \ar@{->}[r]^{f_{2}} \ar@{->}[d]_{f^{*}\phi_{2}} & \mathscr{S}_{2} \ar@{->}[d]^{\phi_{2}} \\
 & \mathscr{S}_{2} \ar@{->}[r]_{\phi_{2}} & \mathscr{C} &  & \mathscr{B} \ar@{->}[r]_{f} & \mathscr{C}
}
\]

Since $h:\mathscr{S}_{1}\overset{\phi_{1}}{\downarrow} \mathscr{C}\to  \mathscr{S}_{2}\overset{\phi_{2}}{\downarrow} \mathscr{C}$ is a morphism of \'{e}tal\'{e}s of residuated lattices, we have the following commutative diagram:
\[
\xymatrix{\kappa_{\phi_{1}} \ar[d]_{\dblp{hp_{1},hp_{2}}} \ar[r]^{\overset{\phi_{1}}{\diamond}} & \mathscr{S}_{1} \ar[d]^h \\ \kappa_{\phi_{2}} \ar[r]_{\overset{\phi_{2}}{\diamond}} & \mathscr{S}_{2}}
\]

To show that $f^*h:f^*\mathscr{S}_{1}\overset{f^*\phi_{1}}{\downarrow}\mathscr{B}\to f^*\mathscr{S}_{2}\overset{f^*\phi_{2}}{\downarrow}\mathscr{B}$ is a morphism of \'{e}tal\'{e}s of residuated lattices, we need to show that the following square is commutative:
\[
\xymatrix{\kappa_{f^{*}\phi_{1}} \ar[d]_{\dblp{f^{*}hp^{*}_{1},f^{*}hp^{*}_{2}}} \ar[r]^{\overset{f^{*}\phi_{1}}{\diamond}} & f^{*}\mathscr{S}_{1} \ar[d]^{f^{*}h} \\ \kappa_{f^{*}\phi_{2}} \ar[r]_{\overset{f^{*}\phi_{2}}{\diamond}} & f^{*}\mathscr{S}_{2}}
\]

Since the composition of the sides in the above square are to $f^{*}\mathscr{S}_{2}$, it suffices to show they are equal when composed with $f^{*}\phi_{2}$ and $f_{2}$.
But first we show $\dblp{hp_{1},hp_{2}}\dblp{f_{1}p^{*}_{1},f_{1}p^{*}_{2}}$ and  $\dblp{f_{2}q^{*}_{1},f_{2}2^{*}_{2}}\dblp{f^{*}hp^{*}_{1},f^{*}hp^{*}_{2}}$ are equal by showing they are equal when composed with both $q_1$ and $q_2$. We have,
\[
\begin{array}{ll}
  q_{1}\dblp{hp_{1},hp_{2}}\dblp{f_{1}p^{*}_{1},f_{1}p^{*}_{2}} & =hp_{1}\dblp{f_{1}p^{*}_{1},f_{1}p^{*}_{2}}\\
   & = hf_{1}p^{*}_{1} \\
   & = f_{2}f^{*}hp^{*}_{1} \\
   & = f_{2}q^{*}_{1}\dblp{f^{*}hp^{*}_{1},f^{*}hp^{*}_{2}}\\
   & = q_{1}\dblp{f_{2}q^{*}_{1},f_{2}2^{*}_{2}}\dblp{f^{*}hp^{*}_{1},f^{*}hp^{*}_{2}}.
\end{array}
\]

Similarly, we can show that,
$$q_{2}\dblp{hp_{1},hp_{2}}\dblp{f_{1}p^{*}_{1},f_{1}p^{*}_{2}}= q_{2}\dblp{f_{2}q^{*}_{1},f_{2}2^{*}_{2}}\dblp{f^{*}hp^{*}_{1},f^{*}hp^{*}_{2}}$$

Now we have,

\[
\begin{array}{ll}
  f^{*}\phi_{2}f^{*}h\overset{f^{*}\phi_{1}}{\diamond} & =f^{*}\phi_{1}\overset{f^{*}\phi_{1}}{\diamond} \\
   & = f^{*}\phi_{1}p^{*}_{1}\\
   & = f^{*}\phi_{2}f^{*}hp^{*}_{1}\\
   & = f^{*}\phi_{2}q^{*}_{1}\dblp{f^{*}hp^{*}_{1},f^{*}hp^{*}_{2}}\\
   & = f^{*}\phi_{2}\overset{f^{*}\phi_{2}}{\diamond}\dblp{f^{*}hp^{*}_{1},f^{*}hp^{*}_{2}},
\end{array}
\]
and
\[
\begin{array}{ll}
  f_{2}f^{*}h\overset{f^{*}\phi_{1}}{\diamond} & =hf_{1}\overset{f^{*}\phi_{1}}{\diamond}\\
   & = h\overset{\phi_{1}}{\diamond}\dblp{f_{1}p^{*}_{1},f_{1}p^{*}_{2}}\\
   & = \overset{\phi_{2}}{\diamond}\dblp{hp_{1},hp_{2}}\dblp{f_{1}p^{*}_{1},f_{1}p^{*}_{2}}\\
   & = \overset{\phi_{2}}{\diamond}\dblp{f_{2}q^{*}_{1},f_{2}q^{*}_{2}}\dblp{f^{*}hp^{*}_{1},f^{*}hp^{*}_{2}} \\
   & =f_{2}\overset{f^{*}\phi_{2}}{\diamond}\dblp{f^{*}hp^{*}_{1},f^{*}hp^{*}_{2}}.
\end{array}
\]

and that concludes the proof.
\end{proof}

The following result shows that the functor of Theorem \ref{pb func on es} can be lifted to residuated lattices.
\begin{theorem}\label{rlefuncy}
Let $f:\mathscr B\longrightarrow \mathscr C$ be a continuous map. Pulling back along $f$ yields a functor $f^*:\textbf{RL-Etale}(\mathscr{C}) \longrightarrow \textbf{RL-Etale}(\mathscr{B})$.
\end{theorem}
\begin{proof}
Follows from Theorem \ref{pb func on es}, Proposition \ref{pb of rles and rles mor} and \citet[p. 35]{johnstone1977topos}.
\end{proof}

\begin{definition}\label{rlespacdef}
Let $\mathscr{B}$ be a topological space. A pair $(\mathscr{B},\mathscr{T_{B}})$ is called an \textit{$RLE$-space} provided that $\mathscr{T_{B}}\overset{\pi_{\mathscr{B}}}{\downarrow}\mathscr{B}$ is an \'{e}tal\'{e} of residuated lattices, for some local homeomorphism $\pi_{\mathscr{B}}:\mathscr{T_{B}}\to \mathscr{B}$. An \textit{inverse morphism} from $RLE$-space $(\mathscr{B},\mathscr{T_{B}})$ to $RLE$-space $(\mathscr{C},\mathscr{T_{C}})$, in short an \textit{$\textbf{RLE}_{inv}$-morphism}, is a pair $(f,\alpha)$, in symbol $(f,\alpha):(\mathscr{B},\mathscr{T_{B}})\longrightarrow (\mathscr{C},\mathscr{T_{C}})$, consisting of a continuous map $f:\mathscr{B}\longrightarrow \mathscr{C}$ and an $RLE_{\mathscr{B}}$-morphism $\alpha:f^{*}\mathscr{T_{C}}\longrightarrow \mathscr{T_{B}}$. The composition of two $\textbf{RLE}_{inv}$-morphisms
$(f,\alpha):(\mathscr{B},\mathscr{T_{B}})\longrightarrow (\mathscr{C},\mathscr{T_{C}})$ and $(g,\beta):(\mathscr{C},\mathscr{T_{C}})\longrightarrow
(\mathscr{D},\mathscr{T_{D}})$ is given by $(gf,\alpha f^{*}\beta\lambda_{f,g,\mathscr{T_{D}}})$, where the morphism used in the composition is given by Diagram \ref{fig:rlespacdeffig0}, in which all squares are pullbacks (so that $g^*\pi_{\mathscr{D}}=\pi_{\mathscr{C}}\beta$, $f^*\pi_{\mathscr{C}}=\pi_{\mathscr{B}}\alpha$, $f^*g^*\pi_{\mathscr{D}}=\pi_{\mathscr{B}}\alpha f^*\beta$ and $(gf)^*\pi_{\mathscr{D}}=\pi_{\mathscr{B}}\alpha f^*\beta\lambda_{f,g,\mathscr{T_{D}}}$) and $\lambda_{f,g,\mathscr{T_{D}}}$ is the unique isomorphism between two pullbacks.
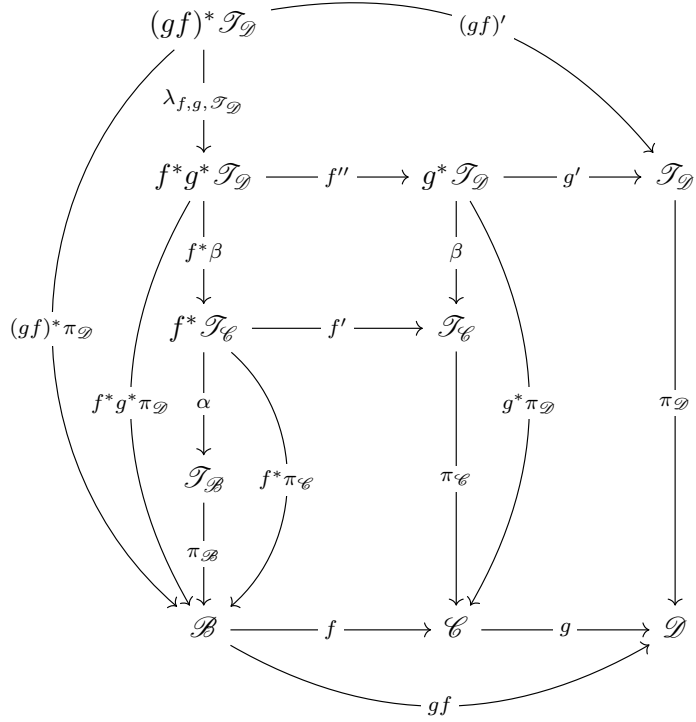
\begin{figure}[h!]
\centering
\begin{tikzcd}
(gf)^{*}\mathscr{T_{D}} \arrow[dd, "{\lambda_{f,g,\mathscr{T_{D}}}}" description] \arrow[rrrrdd, "(gf)'" description, bend left] \arrow[dddddddd, "(gf)^*\pi_{\mathscr{D}}" description, bend right=49] &  &                                                                                                                                                 &  &                                                                           \\
                                                                                                                                                                                                        &  &                                                                                                                                                 &  &                                                                           \\
f^*g^{*}\mathscr{T_{D}} \arrow[dd, "f^*\beta" description] \arrow[rr, "f''" description] \arrow[dddddd, "f^*g^*\pi_{\mathscr{D}}" description, bend right]                                              &  & g^{*}\mathscr{T_{D}} \arrow[dd, "\beta" description] \arrow[rr, "g'" description] \arrow[dddddd, "g^*\pi_{\mathscr{D}}" description, bend left] &  & \mathscr{T}_{\mathscr{D}} \arrow[dddddd, "\pi_{\mathscr{D}}" description] \\
                                                                                                                                                                                                        &  &                                                                                                                                                 &  &                                                                           \\
f^{*}\mathscr{T_{C}} \arrow[dd, "\alpha" description] \arrow[rr, "f'" description] \arrow[dddd, "f^*\pi_{\mathscr{C}}" description, bend left=49]                                                       &  & \mathscr{T}_{\mathscr{C}} \arrow[dddd, "\pi_{\mathscr{C}}" description]                                                                         &  &                                                                           \\
                                                                                                                                                                                                        &  &                                                                                                                                                 &  &                                                                           \\
\mathscr{T}_{\mathscr{B}} \arrow[dd, "\pi_{\mathscr{B}}" description]                                                                                                                                   &  &                                                                                                                                                 &  &                                                                           \\
                                                                                                                                                                                                        &  &                                                                                                                                                 &  &                                                                           \\
\mathscr{B} \arrow[rr, "f" description] \arrow[rrrr, "gf" description, bend right]                                                                                                                      &  & \mathscr{C} \arrow[rr, "g" description]                                                                                                         &  & \mathscr{D}
\end{tikzcd}
\caption{The composition of two $\textbf{RLE}_{inv}$-morphisms}
\label{fig:rlespacdeffig0}
\end{figure}
\end{definition}


In Definition \ref{rlespacdef}, $\alpha$,  $\beta$ and thus $f^*\beta$ are $\textbf{RLE}_{inv}$-morphisms, and $\lambda_{f,g,\mathscr{T_{D}}}$ can be easily verified to be an $\textbf{RLE}_{inv}$-morphism. Thus the composition
$\alpha f^*\beta\lambda_{f,g,\mathscr{T_{D}}}$ is an $\textbf{RLE}_{inv}$-morphism, guaranteeing the well-definedness of the composition.

\begin{theorem}
The class of $RLE$-spaces with $\textbf{RLE}_{inv}$-morphisms, forms a category, denoted by $\textbf{RLE}_{inv}$.
\end{theorem}
\begin{proof}
Letting the identity morphism on $(\mathscr{B},\mathscr{T_{B}})$ to be $(1_{\mathscr B},1_{\mathscr{T}_{\mathscr B}})$, we need to show it acts neutral on both sides. Let $(f,\alpha):(\mathscr{A},\mathscr{T_{A}})\longrightarrow (\mathscr{B},\mathscr{T_{B}})$ be an $\textbf{RLE}_{inv}$-morphism. Using Theorem \ref{rlefuncy}, which implies that $f^*1_{\mathscr{T}_{\mathscr B}}=1_{f^{*}\mathscr{T_{B}}}$, and Definition \ref{rlespacdef}, it follows that $(f,\alpha)(1_{\mathscr B},1_{\mathscr{T}_{\mathscr B}})=(f,\alpha)$ (See Fig. \ref{fig:neutralidenmor}). Analogously, one can show that $(1_{\mathscr B},1_{\mathscr{T}_{\mathscr B}})(f,\alpha)=(f,\alpha)$, for some $\textbf{RLE}_{inv}$-morphism $(f,\alpha):(\mathscr{B},\mathscr{T_{B}})\longrightarrow (\mathscr{C},\mathscr{T_{C}})$.
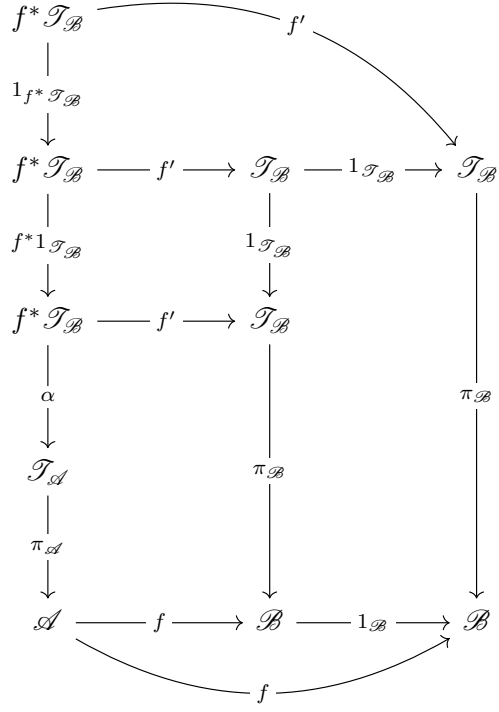
\begin{figure}[h!]
\centering
\begin{tikzcd}
f^{*}\mathscr{T_{B}} \arrow[dd, "1_{f^{*}\mathscr{T_{B}}}" description] \arrow[rrrrdd, "f'" description, bend left] &  &                                                                                                                               &  &                                                                           \\
                                                                                                                    &  &                                                                                                                               &  &                                                                           \\
f^{*}\mathscr{T_{B}} \arrow[dd, "f^*1_{\mathscr{T}_{\mathscr B}}" description] \arrow[rr, "f'" description]         &  & \mathscr{T_{B}} \arrow[dd, "1_{\mathscr{T}_{\mathscr B}}" description] \arrow[rr, "1_{\mathscr{T}_{\mathscr B}}" description] &  & \mathscr{T}_{\mathscr{B}} \arrow[dddddd, "\pi_{\mathscr{B}}" description] \\
                                                                                                                    &  &                                                                                                                               &  &                                                                           \\
f^{*}\mathscr{T_{B}} \arrow[dd, "\alpha" description] \arrow[rr, "f'" description]                                  &  & \mathscr{T_{B}} \arrow[dddd, "\pi_{\mathscr{B}}" description]                                                                 &  &                                                                           \\
                                                                                                                    &  &                                                                                                                               &  &                                                                           \\
\mathscr{T}_{\mathscr{A}} \arrow[dd, "\pi_{\mathscr{A}}" description]                                               &  &                                                                                                                               &  &                                                                           \\
                                                                                                                    &  &                                                                                                                               &  &                                                                           \\
\mathscr{A} \arrow[rr, "f" description] \arrow[rrrr, "f" description, bend right]                                   &  & \mathscr{B} \arrow[rr, "1_{\mathscr{B}}" description]                                                                         &  & \mathscr{B}
\end{tikzcd}
\caption{The identity morphism of $RLE$-space $(\mathscr{B},\mathscr{T_{B}})$}
\label{fig:neutralidenmor}
\end{figure}

To show that composition is associative, let the $\textbf{RLE}_{inv}$-morphisms
\[\xymatrix{(\mathscr{B},\mathscr{T_{B}}) \ar[r]^{(f,\alpha)} & (\mathscr{C},\mathscr{T_{C}}) \ar[r]^{(g,\beta)} & (\mathscr{D},\mathscr{T_{D}}) \ar[r]^{(h,\gamma)} & (\mathscr{E},\mathscr{T_{E}})}\]
be given. Following Diagram \ref{fig:rlespacdeffig}, in which all squares are pullbacks and all triangles commute, it gives the compositions
\begin{center}
  $(g,\beta)(f,\alpha)=(gf,\alpha f^{*}\beta\lambda_{f,g,\mathscr{T_{D}}})$ and $(h,\gamma)(g,\beta)=(hg,\beta g^{*}\gamma\lambda_{g,h,\mathscr{T_{E}}})$.
\end{center}

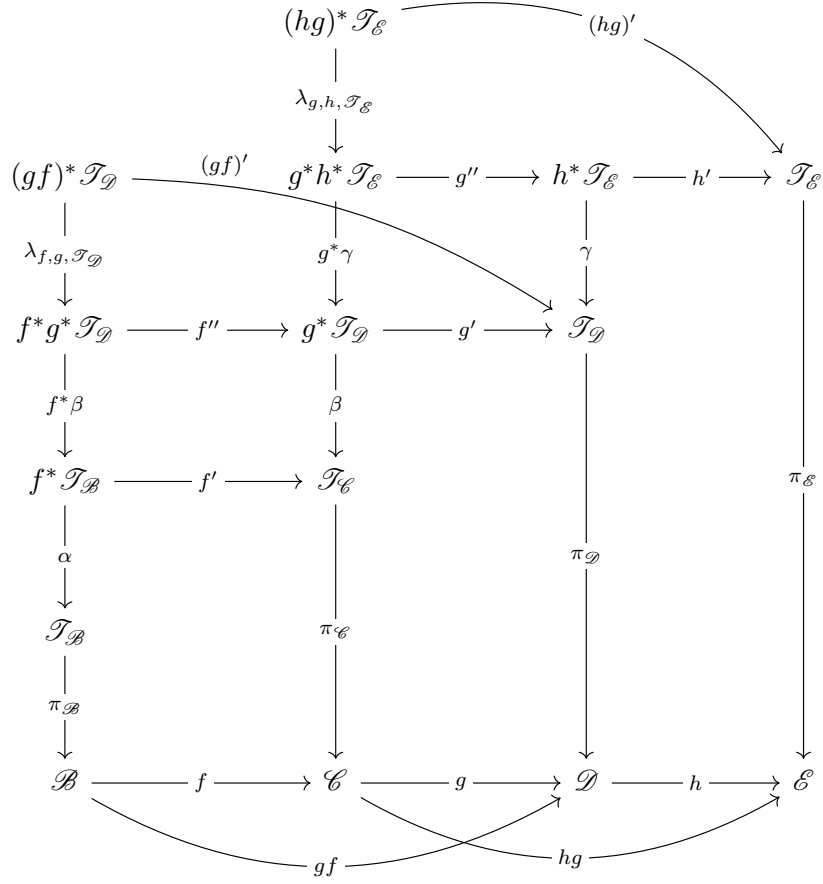
\begin{figure}[h!]
\centering
\begin{tikzcd}
                                                                                                                                               &  & (hg)^{*}\mathscr{T}_{\mathscr{E}} \arrow[rrrrdd, "(hg)'" description, bend left] \arrow[dd, "{\lambda_{g,h,\mathscr{T_{E}}}}" description] &  &                                                                                              &  &                                                                             \\
                                                                                                                                               &  &                                                                                                                                            &  &                                                                                              &  &                                                                             \\
(gf)^{*}\mathscr{T}_{\mathscr{D}} \arrow[dd, "{\lambda_{f,g,\mathscr{T_{D}}}}" description] \arrow[rrrrdd, "    (gf)'" very near start, bend left=15] &  & g^{*}h^{*}\mathscr{T}_{\mathscr{E}} \arrow[rr, "g''" description] \arrow[dd, "g^{*}\gamma" description]                                    &  & h^{*}\mathscr{T}_{\mathscr{E}} \arrow[dd, "\gamma" description] \arrow[rr, "h'" description] &  & \mathscr{T}_{\mathscr{E}} \arrow[dddddddd, "\pi_{\mathscr{E}}" description] \\
                                                                                                                                               &  &                                                                                                                                            &  &                                                                                              &  &                                                                             \\
f^{*}g^{*}\mathscr{T}_{\mathscr{D}} \arrow[dd, "f^*\beta" description] \arrow[rr, "f''" description]                                           &  & g^{*}\mathscr{T}_{\mathscr{D}} \arrow[dd, "\beta" description] \arrow[rr, "g'" description]                                                &  & \mathscr{T}_{\mathscr{D}} \arrow[dddddd, "\pi_{\mathscr{D}}" description]                    &  &                                                                             \\
                                                                                                                                               &  &                                                                                                                                            &  &                                                                                              &  &                                                                             \\
f^{*}\mathscr{T_{B}} \arrow[dd, "\alpha" description] \arrow[rr, "f'" description]                                                             &  & \mathscr{T_{C}} \arrow[dddd, "\pi_{\mathscr{C}}" description]                                                                              &  &                                                                                              &  &                                                                             \\
                                                                                                                                               &  &                                                                                                                                            &  &                                                                                              &  &                                                                             \\
\mathscr{T}_{\mathscr{B}} \arrow[dd, "\pi_{\mathscr{B}}" description]                                                                          &  &                                                                                                                                            &  &                                                                                              &  &                                                                             \\
                                                                                                                                               &  &                                                                                                                                            &  &                                                                                              &  &                                                                             \\
\mathscr{B} \arrow[rr, "f" description] \arrow[rrrr, "gf" description, bend right]                                                             &  & \mathscr{C} \arrow[rr, "g" description] \arrow[rrrr, "hg" description, bend right]                                                         &  & \mathscr{D} \arrow[rr, "h" description]                                                      &  & \mathscr{E}
\end{tikzcd}
\caption{Associativity of the composition of $\textbf{RLE}_{inv}$-morphisms}
\label{fig:rlespacdeffig}
\end{figure}

Following Diagram \ref{fig:rlespacdeffig1}, it verifies that
\begin{center}
  $(h,\gamma)((g,\beta)(f,\alpha))=(h(gf),\alpha f^{*}\beta\lambda_{f,g,\mathscr{T_{D}}}(gf)^*\gamma\lambda_{gf,h,\mathscr{T_{E}}}).$

\end{center}

\begin{figure}[h!]
\centering
\begin{tikzcd}
(h(gf))^{*}\mathscr{T_{E}} \arrow[dd, "{\lambda_{gf,h,\mathscr{T_{E}}}}" description] \arrow[rrrrdd, "(h(gf))'" description, bend left] &  &                                                                                    &  &                                                                           \\
                                                                                                                                        &  &                                                                                    &  &                                                                           \\
(gf)^{*}h^{*}\mathscr{T_{E}} \arrow[dd, "(gf)^*\gamma" description] \arrow[rr, "(gf)''" description]                                    &  & h^{*}\mathscr{T_{E}} \arrow[dd, "\gamma" description] \arrow[rr, "h'" description] &  & \mathscr{T}_{\mathscr{E}} \arrow[dddddd, "\pi_{\mathscr{E}}" description] \\
                                                                                                                                        &  &                                                                                    &  &                                                                           \\
(gf)^{*}\mathscr{T_{D}} \arrow[dd, "{\alpha f^{*}\beta\lambda_{f,g,\mathscr{T_{D}}}}" description] \arrow[rr, "(gf)'" description]      &  & \mathscr{T}_{\mathscr{D}} \arrow[dddd, "\pi_{\mathscr{D}}" description]            &  &                                                                           \\
                                                                                                                                        &  &                                                                                    &  &                                                                           \\
\mathscr{T}_{\mathscr{B}} \arrow[dd, "\pi_{\mathscr{B}}" description]                                                                   &  &                                                                                    &  &                                                                           \\
                                                                                                                                        &  &                                                                                    &  &                                                                           \\
\mathscr{B} \arrow[rr, "gf" description]                                                                                                &  & \mathscr{D} \arrow[rr, "h" description]                                            &  & \mathscr{E}
\end{tikzcd}
\caption{The composition of two $\textbf{RLE}_{inv}$-morphisms $gf$ and $h$}
\label{fig:rlespacdeffig1}
\end{figure}

Following Diagram \ref{fig:rlespacdeffig2}, it verifies that
\begin{center}
  $((h,\gamma)(g,\beta))(f,\alpha)=((hg)f,\alpha f^{*}(\beta g^{*}\gamma\lambda_{g,h,\mathscr{T_E}})\lambda_{f,hg,\mathscr{T_{E}}}).$

\end{center}

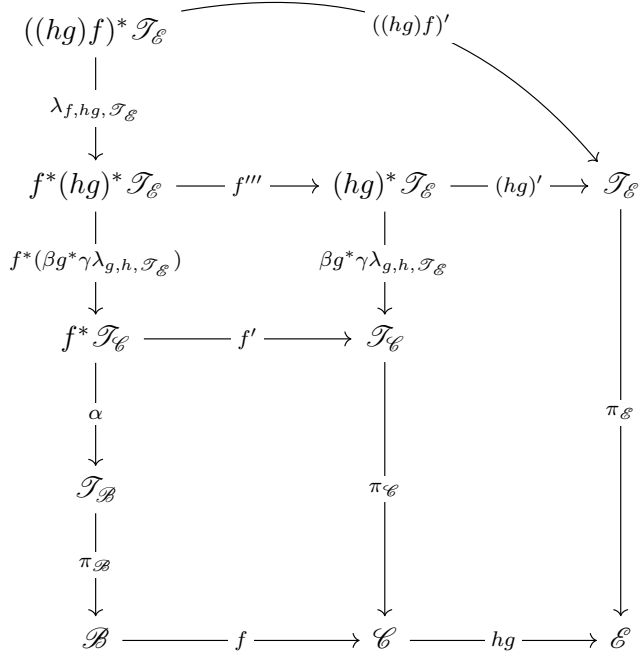
\begin{figure}[h!]
\centering
\begin{tikzcd}
((hg)f)^{*}\mathscr{T_{E}} \arrow[dd, "{\lambda_{f,hg,\mathscr{T_{E}}}}" description] \arrow[rrrrdd, "((hg)f)'" description, bend left]       &  &                                                                                                                                    &  &                                                                           \\
                                                                                                                                              &  &                                                                                                                                    &  &                                                                           \\
f^{*}(hg)^{*}\mathscr{T_{E}} \arrow[rr, "f'''" description] \arrow[dd, "{f^{*}(\beta g^{*}\gamma\lambda_{g,h,\mathscr{T_E}})}" description] &  & (hg)^{*}\mathscr{T_{E}} \arrow[dd, "{\beta g^{*}\gamma\lambda_{g,h,\mathscr{T_E}}}" description] \arrow[rr, "(hg)'" description] &  & \mathscr{T}_{\mathscr{E}} \arrow[dddddd, "\pi_{\mathscr{E}}" description] \\
                                                                                                                                              &  &                                                                                                                                    &  &                                                                           \\
f^{*}\mathscr{T_{C}} \arrow[dd, "\alpha" description] \arrow[rr, "f'" description]                                                            &  & \mathscr{T}_{\mathscr{C}} \arrow[dddd, "\pi_{\mathscr{C}}" description]                                                            &  &                                                                           \\
                                                                                                                                              &  &                                                                                                                                    &  &                                                                           \\
\mathscr{T}_{\mathscr{B}} \arrow[dd, "\pi_{\mathscr{B}}" description]                                                                         &  &                                                                                                                                    &  &                                                                           \\
                                                                                                                                              &  &                                                                                                                                    &  &                                                                           \\
\mathscr{B} \arrow[rr, "f" description]                                                                                                       &  & \mathscr{C} \arrow[rr, "hg" description]                                                                                           &  & \mathscr{E}
\end{tikzcd}
\caption{The composition of two $\textbf{RLE}_{inv}$-morphisms $f$ and $hg$}
\label{fig:rlespacdeffig2}
\end{figure}

To show that $(h,\gamma)((g,\beta)(f,\alpha))=((h,\gamma)(g,\beta))(f,\alpha)$, we know $h(gf)=(hg)f$, so we need to show
$\alpha f^{*}\beta\lambda_{f,g,\mathscr{T_{D}}}(gf)^*\gamma\lambda_{gf,h,\mathscr{T_{E}}}=\alpha f^{*}(\beta g^{*}\gamma\lambda_{g,h,\mathscr{T_E}})\lambda_{f,hg,\mathscr{T_{E}}}$, which holds if the equality
\begin{equation}\label{e1}
  f^{*}\beta\lambda_{f,g,\mathscr{T_{D}}}(gf)^*\gamma\lambda_{gf,h,\mathscr{T_{E}}}=f^{*}(\beta g^{*}\gamma\lambda_{g,h,\mathscr{T_E}})\lambda_{f,hg,\mathscr{T_{E}}}
\end{equation}
holds. Codomains of two morphisms in the equality \ref{e1} are $f^*(\mathscr{T_{B}})$, so by a pullback in Diagram \ref{fig:rlespacdeffig}, Equation \ref{e1} holds if and only if equalities
\begin{equation}\label{e11}
  \pi_{\mathscr{B}}\alpha f^{*}\beta\lambda_{f,g,\mathscr{T_{D}}}(gf)^*\gamma\lambda_{gf,h,\mathscr{T_{E}}}=\pi_{\mathscr{B}}\alpha f^{*}(\beta g^{*}\gamma\lambda_{g,h,\mathscr{T_E}})\lambda_{f,hg,\mathscr{T_{E}}}
\end{equation}
and
\begin{equation}\label{e12}
  f' f^{*}\beta\lambda_{f,g,\mathscr{T_{D}}}(gf)^*\gamma\lambda_{gf,h,\mathscr{T_{E}}}=f' f^{*}(\beta g^{*}\gamma\lambda_{g,h,\mathscr{T_E}})\lambda_{f,hg,\mathscr{T_{E}}}
\end{equation}
hold. For Equation \ref{e11} we have,
\[
\begin{array}{ll}
  \pi_{\mathscr{B}}\alpha f^{*}\beta\lambda_{f,g,\mathscr{T_{D}}}(gf)^*\gamma\lambda_{gf,h,\mathscr{T_{E}}} & \\ =(h(gf))^*\pi_{\mathscr{E}}& \mbox{ by the outer pullback square in Diagram \ref{fig:rlespacdeffig1}}\\
=((hg)f)^*\pi_{\mathscr{E}}&  \mbox{ by associativity of morphisms in $\textbf{Top}$}\\
=\pi_{\mathscr{B}}\alpha f^{*}(\beta g^{*}\gamma\lambda_{g,h,\mathscr{T_E}})\lambda_{f,hg,\mathscr{T_{E}}}.&  \mbox{  by the outer pullback square in Diagram \ref{fig:rlespacdeffig2}}
\end{array}
\]

For Equation \ref{e12} the left side is
\[
\begin{array}{ll}
  f' f^{*}\beta\lambda_{f,g,\mathscr{T_{D}}}(gf)^*\gamma\lambda_{gf,h,\mathscr{T_{E}}} & \\
 =\beta f'' \lambda_{f,g,\mathscr{T_{D}}}(gf)^*\gamma\lambda_{gf,h,\mathscr{T_{E}}},& \mbox{ by the outer pullback square in Diagram \ref{fig:rlespacdeffig}}
\end{array}
\]

and the right side is
\[
\begin{array}{ll}
  f' f^{*}(\beta g^{*}\gamma\lambda_{g,h,\mathscr{T_E}})\lambda_{f,hg,\mathscr{T_{E}}} & \\
 =\beta g^{*}\gamma\lambda_{g,h,\mathscr{T_E}}f'''\lambda_{gf,h,\mathscr{T_{E}}}.& \mbox{ by the outer pullback square in Diagram \ref{fig:rlespacdeffig2}}
\end{array}
\]

So Equation \ref{e12} holds if the equality
\begin{equation}\label{e2}
  f'' \lambda_{f,g,\mathscr{T_{D}}}(gf)^*\gamma\lambda_{gf,h,\mathscr{T_{E}}}= g^{*}\gamma\lambda_{g,h,\mathscr{T_E}}f'''\lambda_{gf,h,\mathscr{T_{E}}}
\end{equation}
holds. Codomains of the two morphisms in Equation \ref{e2} are $g^*\mathscr{T_{D}}$, so by a pullback in Diagram \ref{fig:rlespacdeffig}, Equation \ref{e2} holds if and only if equalities
\begin{equation}\label{e21}
 \pi_{\mathscr{C}}\beta f'' \lambda_{f,g,\mathscr{T_{D}}}(gf)^*\gamma\lambda_{gf,h,\mathscr{T_{E}}}= \pi_{\mathscr{C}}\beta g^{*}\gamma\lambda_{g,h,\mathscr{T_E}}f'''\lambda_{gf,h,\mathscr{T_{E}}}
\end{equation}
and
\begin{equation}\label{e22}
  g' f'' \lambda_{f,g,\mathscr{T_{D}}}(gf)^*\gamma\lambda_{gf,h,\mathscr{T_{E}}}= g'g^{*}\gamma\lambda_{g,h,\mathscr{T_E}}f'''\lambda_{gf,h,\mathscr{T_{E}}}
\end{equation}
hold. For Equation \ref{e21} we have,
\[
\begin{array}{ll}
  \pi_{\mathscr{C}}\beta f'' \lambda_{f,g,\mathscr{T_{D}}}(gf)^*\gamma\lambda_{gf,h,\mathscr{T_{E}}} & \\
=f\pi_{\mathscr{B}}\alpha f^{*}\beta \lambda_{f,g,\mathscr{T_{D}}}(gf)^*\gamma\lambda_{gf,h,\mathscr{T_{E}}}& \mbox{ by a pullback square in Diagram \ref{fig:rlespacdeffig}}\\
=f(h(gf))^{*}\pi_{\mathscr{E}}&  \mbox{ by the outer pullback square in Diagram \ref{fig:rlespacdeffig1}}\\
=f((hg)f)^{*}\pi_{\mathscr{E}}&  \mbox{  by associativity of morphisms in $\textbf{Top}$}\\
=f\pi_{\mathscr{B}}\alpha f^{*}(\beta g^{*}\gamma\lambda_{g,h,\mathscr{T_{E}}})\lambda_{f,hg,\mathscr{T_{E}}} &  \mbox{ by the outer pullback square in Diagram \ref{fig:rlespacdeffig2}}\\
=\pi_{\mathscr{C}}f' f^{*}(\beta g^{*}\gamma\lambda_{g,h,\mathscr{T_{E}}})\lambda_{f,hg,\mathscr{T_{E}}} &  \mbox{ by a pullback square in Diagram \ref{fig:rlespacdeffig}}\\
=\pi_{\mathscr{C}}\beta g^{*}\gamma\lambda_{g,h,\mathscr{T_{E}}} f'''\lambda_{f,hg,\mathscr{T_{E}}}. &  \mbox{ by a pullback square in Diagram \ref{fig:rlespacdeffig2}}\\
\end{array}
\]

For Equation \ref{e22} the left side is
\[
\begin{array}{ll}
g' f'' \lambda_{f,g,\mathscr{T_{D}}}(gf)^*\gamma\lambda_{gf,h,\mathscr{T_{E}}} & \\
=(gf)'(gf)^*\gamma\lambda_{gf,h,\mathscr{T_{E}}}& \mbox{ by a commutative triangle in Diagram \ref{fig:rlespacdeffig}}\\
=\gamma(gf)''\lambda_{gf,h,\mathscr{T_{E}}},&  \mbox{ by a pullback square in Diagram \ref{fig:rlespacdeffig1}}
\end{array}
\]
and the right side is
\[
\begin{array}{ll}
g'g^{*}\gamma\lambda_{g,h,\mathscr{T_E}}f'''\lambda_{f,hg,\mathscr{T_{E}}} & \\
=\gamma g''\lambda_{g,h,\mathscr{T_{E}}}f'''\lambda_{f,hg,\mathscr{T_{E}}}.& \mbox{ by a pullback square in Diagram \ref{fig:rlespacdeffig}}
\end{array}
\]

So Equation \ref{e22} holds if the equality
\begin{equation}\label{e3}
  (gf)''\lambda_{gf,h,\mathscr{T_{E}}}=g''\lambda_{g,h,\mathscr{T_{E}}}f'''\lambda_{f,hg,\mathscr{T_{E}}}
\end{equation}
holds. Codomains of the two morphisms in the equality \ref{e3} are $h^*\mathscr{T_{E}}$, so by a pullback in Diagram \ref{fig:rlespacdeffig}, Equation \ref{e3} holds if and only if the equalities
\begin{equation}\label{e31}
  \pi_{\mathscr{D}}\gamma(gf)''\lambda_{gf,h,\mathscr{T_{E}}}=\pi_{\mathscr{D}}\gamma g''\lambda_{g,h,\mathscr{T_{E}}}f'''\lambda_{f,hg,\mathscr{T_{E}}}
\end{equation}
and
\begin{equation}\label{e32}
  h'(gf)''\lambda_{gf,h,\mathscr{T_{E}}}=h'g''\lambda_{g,h,\mathscr{T_{E}}}f'''\lambda_{f,hg,\mathscr{T_{E}}}
\end{equation}
hold. For Equation \ref{e31} we have,
\[
\begin{array}{ll}
  \pi_{\mathscr{D}}\gamma(gf)''\lambda_{gf,h,\mathscr{T_{E}}} & \\
=gf\pi_{\mathscr{B}}\alpha f^{*}\beta \lambda_{f,g,\mathscr{T_{D}}}(gf)^*\gamma\lambda_{gf,h,\mathscr{T_{E}}}& \mbox{ by a pullback square in Diagram \ref{fig:rlespacdeffig1}}\\
=gf(h(gf))^{*}\pi_{\mathscr{E}}&  \mbox{ by the outer pullback square in Diagram \ref{fig:rlespacdeffig1}}\\
=gf((hg)f)^{*}\pi_{\mathscr{E}}&  \mbox{  by associativity of morphisms in $\textbf{Top}$}\\
=gf\pi_{\mathscr{B}}\alpha f^{*}(\beta g^{*}\gamma\lambda_{g,h,\mathscr{T_{E}}})\lambda_{f,hg,\mathscr{T_{E}}} &  \mbox{ by the outer pullback square in Diagram \ref{fig:rlespacdeffig2}}\\
=g\pi_{\mathscr{C}}f' \beta g^{*}\gamma\lambda_{g,h,\mathscr{T_{E}}}f'''\lambda_{f,hg,\mathscr{T_{E}}} &  \mbox{ by a pullback square in Diagram \ref{fig:rlespacdeffig2}}\\
=\pi_{\mathscr{D}}\gamma g'' \lambda_{g,h,\mathscr{T_{E}}} f'''\lambda_{f,hg,\mathscr{T_{E}}}. &  \mbox{ by a pullback square in Diagram \ref{fig:rlespacdeffig}}\\
\end{array}
\]

For Equation \ref{e32} we have,
\[
\begin{array}{ll}
  h'(gf)''\lambda_{gf,h,\mathscr{T_{E}}} & \\
=(h(gf))'& \mbox{ by a commutative triangle in Diagram \ref{fig:rlespacdeffig1}}\\
=((hg)f)'&  \mbox{ by associativity of morphisms in $\textbf{Top}$}\\
=(hg)'f'''\lambda_{f,hg,\mathscr{T_{E}}}&  \mbox{  by a commutative triangle in Diagram \ref{fig:rlespacdeffig2}}\\
=h'g''\lambda_{g,h,\mathscr{T_{E}}}f'''\lambda_{f,hg,\mathscr{T_{E}}}. &  \mbox{ by a commutative triangle in Diagram \ref{fig:rlespacdeffig}}
\end{array}
\]

This concludes the result.
\end{proof}
At the end, we define a contravariant functor, called the section functor, from the category $\textbf{RLE}_{inv}$ to the category of residuated lattices.
\begin{theorem}\label{maintheoremofart}
The mapping $\mathcal{S}:RLE_{inv} \longrightarrow \textsc{RL}$, sending an $RLE$-space $(\mathscr{B},\mathscr{T_{B}})$ to the residuated lattice $\Gamma(\mathscr{B},\mathscr{T_{B}})$ and an $\textbf{RLE}_{inv}$-morphism $(f,\alpha):(\mathscr{B},\mathscr{T_{B}})\longrightarrow (\mathscr{C},\mathscr{T_{C}})$ to the residuated lattice morphism $\mathcal{S}(f,\alpha):\Gamma(\mathscr{C},\mathscr{T_{C}})\longrightarrow \Gamma(\mathscr{B},\mathscr{T_{B}})$, given by $\mathcal{S}(f,\alpha)(\sigma)(b)=\alpha(b,\sigma(f(b)))$, is a contravariant functor.
\end{theorem}
\begin{proof}
Let $(\mathscr{B},\mathscr{T_{B}})$ be an $RLE$-space. By Definition \ref{bund-etal of RL} $\Gamma(\mathscr{B},\mathscr{T_{B}})$ is a residuated lattice. Let $(f,\alpha):(\mathscr{B},\mathscr{T_{B}})\longrightarrow (\mathscr{C},\mathscr{T_{C}})$ be an $\textbf{RLE}_{inv}$-morphism. For a given global section $\sigma$ of $\mathscr{T_{C}}$, with a little effort, one can see that $\mathcal{S}(f,\alpha)(\sigma)$ is a global section of $\mathscr{T_{B}}$. Let $\sigma$ and $\tau$ be two global section of $\mathscr{T_{C}}$. For an arbitrary $b\in \mathscr{B}$, We have
\[
\begin{array}{ll}
\mathcal{S}(f,\alpha)(\sigma\vee \tau)(b) &=\alpha(b,\sigma\vee \tau(f(b))) \\
&=\alpha(b,\sigma(f(b))\vee_{f(b)} \tau(f(b))) \\
&=\alpha|_{b}(b,\sigma(f(b))\vee_{f(b)} \tau(f(b))) \\
&=\alpha|_{b}(b,\sigma(f(b))\vee_{f(b)} \tau(f(b))) \\
&=\alpha|_{b}((b,\sigma(f(b)))\vee_{b} (b,\tau(f(b)))) \\
&=\alpha|_{b}((b,\sigma(f(b))))\vee_{b} \alpha|_{b}((b,\tau(f(b)))) \\
&=\alpha((b,\sigma(f(b))))\vee_{b} \alpha((b,\tau(f(b)))) \\
&=\mathcal{S}(f,\alpha)(\sigma)(b)\vee_{b} \mathcal{S}(f,\alpha)(\tau)(b)\\
&=(\mathcal{S}(f,\alpha)(\sigma)\vee \mathcal{S}(f,\alpha)(\tau))(b).
\end{array}
\]

Analogously, we can show that the function $\mathcal{S}(f,\alpha)$ preserves the other fundamental operations, and so it is an $RL$-morphism. Also, it is obvious that $\mathcal{S}(1_{\mathscr B},1_{\mathscr{T}_{\mathscr B}})=1_{\Gamma(\mathscr{B},\mathscr{T_{B}})}$. Next we show that $\mathcal{S}$ preserves the reverse composition, let the $\textbf{RLE}_{inv}$-morphisms
\[\xymatrix{(\mathscr{B},\mathscr{T_{B}}) \ar[r]^{(f,\alpha)} & (\mathscr{C},\mathscr{T_{C}}) \ar[r]^{(g,\beta)} & (\mathscr{D},\mathscr{T_{D}})}\]
be given. To show that $\mathcal{S}((g,\beta)(f,\alpha))=\mathcal{S}(f,\alpha)\mathcal{S}(g,\beta)$, we need to show for an arbitrary section $\sigma$ of $\mathscr{T_{D}}$ and $b\in \mathscr{B}$,
\begin{equation}\label{e5}
  \mathcal{S}((g,\beta)(f,\alpha))(\sigma)(b)=(\mathcal{S}(f,\alpha)\mathcal{S}(g,\beta))(\sigma)(b)
\end{equation}

By definition we have,
\[
\begin{array}{ll}
\mathcal{S}((g,\beta)(f,\alpha))(\sigma)(b) &=\mathcal{S}(gf,\alpha f^{*}\beta\lambda_{f,g,\mathscr{T_{D}}})(\sigma)(b) \\
&=\alpha f^{*}\beta\lambda_{f,g,\mathscr{T_{D}}}(b,\sigma(gf(b)))
\end{array}
\]
and
\[
\begin{array}{ll}
(\mathcal{S}(f,\alpha)\mathcal{S}(g,\beta))(\sigma)(b) &=\mathcal{S}(f,\alpha)(\mathcal{S}(g,\beta)(\sigma))(b) \\
&=\alpha(b,(\mathcal{S}(g,\beta)(\sigma)(f(b))))\\
&=\alpha(b,\beta(f(b),\sigma(gf(b))))\\
\end{array}
\]

So the equation \ref{e5} holds if the equality
\begin{equation}\label{e6}
  f^{*}\beta\lambda_{f,g,\mathscr{T_{D}}}(b,\sigma(gf(b)))=(b,\beta(f(b),\sigma(gf(b))))
\end{equation}
holds. By a pullback in Diagram \ref{fig:rlespacdeffig0}, Equation \ref{e6} holds if and only if equalities
\begin{equation}\label{e61}
  f^{*}\pi_{\mathscr{C}}f^{*}\beta\lambda_{f,g,\mathscr{T_{D}}}(b,\sigma(gf(b)))=f^{*}\pi_{\mathscr{C}}(b,\beta(f(b),\sigma(gf(b))))
\end{equation}
and
\begin{equation}\label{e62}
 f'f^{*}\beta\lambda_{f,g,\mathscr{T_{D}}}(b,\sigma(gf(b)))=f'(b,\beta(f(b),\sigma(gf(b))))
\end{equation}
hold. For Equation \ref{e61} we have,
\[
\begin{array}{ll}
 f^{*}\pi_{\mathscr{C}}f^{*}\beta\lambda_{f,g,\mathscr{T_{D}}}(b,\sigma(gf(b))) & \\
=(gf)^{*}\pi_{\mathscr{D}}(b,\sigma(gf(b)))& \mbox{by a commutative triangle in Diagram \ref{fig:rlespacdeffig0}}\\
=b&  \mbox{$(gf)^{*}\pi_{\mathscr{D}}$ is a projection on first argument}\\
=f^{*}\pi_{\mathscr{C}}(b,\beta(f(b),\sigma(gf(b)))).&  \mbox{$f^{*}\pi_{\mathscr{C}}$ is a projection on first argument}
\end{array}
\]

For Equation \ref{e62} the left side is
\[
\begin{array}{ll}
  f'f^{*}\beta\lambda_{f,g,\mathscr{T_{D}}}(b,\sigma(gf(b))) & \\
 =\beta f'' \lambda_{f,g,\mathscr{T_{D}}}(b,\sigma(gf(b))),& \mbox{ by a pullback square in Diagram \ref{fig:rlespacdeffig0}}
\end{array}
\]
and the right side is
\[
\begin{array}{ll}
  f'(b,\beta(f(b),\sigma(gf(b)))) & \\
 =\beta(f(b),\sigma(gf(b))).& \mbox{$f'$ is a projection on second argument}
\end{array}
\]

So Equation \ref{e62} holds if the equality
\begin{equation}\label{e7}
  f'' \lambda_{f,g,\mathscr{T_{D}}}(b,\sigma(gf(b)))=(f(b),\sigma(gf(b)))
\end{equation}
holds. By a pullback in Diagram \ref{fig:rlespacdeffig0}, Equation \ref{e7} holds if and only if equalities
\begin{equation}\label{e71}
  g^{*}\pi_{\mathscr{D}}f'' \lambda_{f,g,\mathscr{T_{D}}}(b,\sigma(gf(b)))=g^{*}\pi_{\mathscr{D}}(f(b),\sigma(gf(b)))
\end{equation}
and
\begin{equation}\label{e72}
 g'f'' \lambda_{f,g,\mathscr{T_{D}}}(b,\sigma(gf(b)))=g'(f(b),\sigma(gf(b)))
\end{equation}
hold.

For Equation \ref{e71} we have,
\[
\begin{array}{ll}
 g^{*}\pi_{\mathscr{D}}f'' \lambda_{f,g,\mathscr{T_{D}}}(b,\sigma(gf(b))) & \\
=f(gf)^{*}\pi_{\mathscr{D}}(b,\sigma(gf(b)))& \mbox{by a pullback square in Diagram \ref{fig:rlespacdeffig0}}\\
=f(b)&  \mbox{$(gf)^{*}\pi_{\mathscr{D}}$ is a projection on the first argument}\\
=g^{*}\pi_{\mathscr{D}}(f(b),\sigma(gf(b))).&  \mbox{$g^{*}\pi_{\mathscr{D}}$ is a projection on first argument}
\end{array}
\]

For Equation \ref{e72} we have
\[
\begin{array}{ll}
  g'f'' \lambda_{f,g,\mathscr{T_{D}}}(b,\sigma(gf(b))) & \\
 =(gf)'(b,\sigma(gf(b)))& \mbox{ by a commutative triangle in Diagram \ref{fig:rlespacdeffig0}}\\
 =\sigma(gf(b))& \mbox{$(gf)'$ is a projection on the second argument}\\
 =g'(f(b),\sigma(gf(b))).& \mbox{$g'$ is a projection on the second argument}
\end{array}
\]
This concludes the result.
\end{proof}



\end{document}